\documentclass[aop,preprint]{imsart}

\RequirePackage[OT1]{fontenc}
\RequirePackage{amsthm,amsmath,amssymb,enumerate}
\RequirePackage[authoryear]{natbib}
\RequirePackage[colorlinks,citecolor=blue,urlcolor=blue]{hyperref}
\usepackage{algorithm}
\usepackage{algcompatible}
\usepackage{caption}
\usepackage{subcaption}
\usepackage{graphicx}
\usepackage{geometry}
\usepackage{array}

\usepackage{changes}
\definechangesauthor[name={Guilherme Ost}, color=red]{GO}
\usepackage{soul,color}
\soulregister\cite7
\soulregister\ref7
\soulregister\pageref7

\startlocaldefs
\numberwithin{equation}{section}
\theoremstyle{plain}
\newtheorem{thm}{Theorem}
\newtheorem{lemma}{Lemma}
\newtheorem{corollary}{Corollary}				
\newtheorem{proposition}{Proposition}
\theoremstyle{definition}
\newtheorem{definition}{Definition}
\newtheorem{assumption}{Assumption}
\newtheorem{remark}{Remark}

\def\E{{\mathbb E}}

\def\P{{\mathbb P}}

\def\R{{\mathbb R}}
\def\Z{{\mathbb Z}}

\def\P{{\mathbb P}}
\def\N{{\mathbb N}}

\def\cF {\mathcal{F}}

\def\cV{\mathcal V}

\def\bT{\mathbb T}

\endlocaldefs

\begin{document}

\begin{frontmatter}
\title{Sparse space-time models: Concentration Inequalities and Lasso 
} 
\runtitle{}

\begin{aug}
\author{\fnms{G.} \snm{Ost}\thanksref{m1}
\ead[label=e2]{guilhermeost@gmail.com}}
\and
\author{\fnms{P.} \snm{Reynaud-Bouret}\thanksref{m2}\ead[label=e2]{reynaudb@unice.fr}} 

\thankstext{t1}{\today}
\runauthor{G. Ost and P. Reynaud-Bouret}

\affiliation{Universidade Federal do Rio de Janeiro, Brasil \thanksmark{m1} and Universit\'e C\^ote d'Azur, CNRS, LJAD, France\thanksmark{m2}}


\end{aug}

\begin{abstract}
Inspired by Kalikow-type decompositions, we introduce a new stochastic model of infinite neuronal networks, for which we establish sharp oracle inequalities for Lasso methods and restricted eigenvalue properties for the associated Gram matrix with high probability. These results hold even if the network is only partially observed.  The main argument  rely on the fact that concentration inequalities can easily be derived whenever the transition probabilities of the underlying process admit a {\it sparse space-time representation}.
\end{abstract}



\end{frontmatter}
\newcommand{\pat}[1]{\textcolor{magenta}{#1}}

\section{Introduction}
Lasso-type methods in classic regression settings assume that the corresponding Gram matrix  $G$  fulfills nice properties such as the Restricted Isometry Property (RIP), Restricted Eigenvalue (RE) conditions, etc. In many works  (see for instance \cite{BvdG:09,BvdG:11, vdG:16}
and references therein),  the explanatory variables involved in the Gram matrix are given at first and it is natural to define the regression model conditionally to these variables. In this sense, it is also natural to assume such properties on the Gram matrix $G$ without trying to show that they are fulfilled with high probability. In practice, it is computationally difficult to check whether $G$ satisfies or not these assumptions and many works have shown how random matrices can fulfill such properties with high probability (see for instance \cite{CT:05, RV:08} or the references in \cite{Tropp:15}).

However, in several probabilistic frameworks it is difficult to separate the study of the Gram matrix from the study of the process itself (see for instance \cite{KC:15} where the probabilistic framework is of auto-regressive kind or \cite{GM:19} for Markovian continuous processes). Several works have therefore shown that with high probability, Lasso or other adaptive methods satisfy  oracle inequalities or minimax results and that, on the same event, the corresponding Gram matrix (which is considered here also as a random variable) satisfies RIP or RE (see for instance \cite{KC:15, BM:15, JRW:15, HRBRSW:18}).

We are here interested in a particular type of stochastic process that can model the spiking events of a possibly infinite network of neurons. 
Many probabilistic models of neuronal activities in a network exist. As examples, let us cite continuous frameworks where both the voltage and the spiking activity of each neuron are modeled (see for instance \cite{SG:13}), or where the spike trains are directly modeled by point processes (see e.g \cite{chevallier} ). There are also approaches, closer to the present one, where discrete time is used (see \cite{CC:14,EvaGalves:15}).
 
Although many statistical methods have been developed to estimate the probability to spike given the past for various probabilistic models, most of the time one assumes that the network is fully observed (see for instance \cite{PC:09} on Wold processes, \cite{CSSW:17} on  Hawkes processes or \cite{MRW:18} on Poisson counts). 
Let us underline the work by \citep{hrbr}, which is the closest to ours from a statistical point of view, in which the authors applied a Lasso method on point processes and derived an oracle inequality on an event where the corresponding Gram matrix $G$ is invertible. In a second step, the authors have shown that $G$ is invertible with high probability when the observed process is a linear  Hawkes process and the small fixed number of   observed  neurons correspond in fact to the totality of the network (see also \cite{KSKL:10} for an application on real data of Lasso-type methods). 
However, in practice, data biologists record are much more scarce than a complete recordings of the whole network  activity. Most of the time, just few tens of neurons are recorded and they correspond to neurons that are embedded in at least a network of thousands of neurons.  
 
We are aware of only two articles which clearly deal more deeply with the problem of partial observation from a mathematical point of view : \cite{LT:16} and \cite{AAEG:16}.  In \cite{LT:16}, the authors assume that the configuration describing the neural activity 
follows a Gibbsian distribution, which does not take dependency in time into account, fact which is of the utmost importance in neuroscience. 
In \cite{AAEG:16}, the interaction neighborhood of a given neuron is estimated by assuming that we observe more neurons than this neighborhood even if it is not the totality of network. In practice, the complexity of the algorithm makes it difficult to apply it on large data sets. In these two works, the purpose is clearly to deal with the fact that some neurons (or nodes of the network)  are not recorded at all whereas the recordings are complete for the observed neurons. In a slightly different statistical setting, note also the work by  \cite{MRW:19}, which  deals with another kind of missing data, the applied filter being i.i.d. in both time and nodes but without nodes that are completely  unrecorded.
 
The aim of the present work is to show that  with high probability, the Gram matrix $G$ satisfies nice properties such as invertibility or RE condition with as few assumptions as possible on the underlying probabilistic models and this even if we observe only a small number of neurons embedded in an infinite neuronal network , for which most of the neurons are not observed at all.

Inspired by \cite{GL:13} model and Kalikow-type decompositions (see \cite{Kalikow:90, GGLO:13}), we consider discrete time models for which the probability of a neuron to spike at a given time unit given the past configuration may only depend on a few neurons (assimilated to space positions) and few time steps. Hence the dependencies  in time and space should be very small but may be random and chosen at each step. We use this {\it probabilistic sparsity} in time and space to prove concentration inequalities for various functionals including Gram matrices. We employ these concentration inequalities to prove that RE properties are satisfied with large probability even on a partially observed infinite network. Therefore we show that  Lasso  methods have good theoretical properties even in this case.

The paper is organized as follows. In Section \ref{not}, we provide the main notation and  present briefly the stochastic framework with its main assumptions. In Section \ref{lass}, we prove an oracle inequality for the Lasso estimator of the transition probabilities of this model. The oracle inequality is derived on a certain event on which some properties of the Gram matrix are met. Examples of useful dictionaries are also presented in this section. We discuss, in Section \ref{mod}, the definition of space-time decomposition and show examples of discrete time models where it applies.
Besides, thanks to the definition of a simulation algorithm inspired by \cite{GL:13}, we prove that stochastic models admitting a space-time decomposition have a stationary version even on infinite networks under some conditions of {\it probabilistic sparsity}. We prove that these conditions are usually much less stringent than the ones of the literature. Still in the same section, we obtain concentration inequalities for such processes, under some additional exponential constraints, by adapting arguments of \cite{V:97}. This allows us to prove that the introduced Gram matrices based on a partial observation of the network  are invertible or satisfy RE with high probability. This is done in Section \ref{back}. A brief conclusion is given in Section \ref{conclusion}. All proofs are given in the Section \ref{Sec:proof}.

\section{Stochastic framework and notation}
\label{not}


We write $\N$ to denote the set of natural numbers $\{0,1,2,\ldots\}$. The sets of integers $\{\ldots,-1,0,1,\ldots\}$ is denoted by $\Z$. 
The set of strictly negative and positive integers are denoted by $\Z_{-}$ and $\Z_{+}$ respectively. 

We consider a stationary stochastic chain ${\bf X}=(X_{i,t} )_{i\in I, t \in \Z}$ taking values in $\{0,1\}^{I\times\Z}$, defined on a probability space $ ( \Omega , \cF , \mathbb{P} )$,
where $I$ is a countable (possibly infinite) set. The set $I$ represents the set of neurons in the network. For each $i\in I$ and $t\in\Z$, 
$$
X_{i,t}=
\begin{cases}
1, \   \mbox{if neuron} \ i \ \mbox{spikes at time} \ t, \\
0 , \ \mbox{otherwise}.
\end{cases}
$$
The configuration of ${\bf X}$ at time $t\in\Z$ is denoted by $X_{t}=(X_{i,t}, i\in I).$ 
For $s,t\in \Z$  with $s<t$, $X_{i,s:t}$ stands for the collection $(X_{i,s},\ldots, X_{i,t})$ and $X_{s:t}$ for the collection $(X_{i,r} )_{ i\in I ,s\leq r\leq t}$. For each $t\in\Z$, $X_{-\infty:t}$ denotes the past history $(\ldots,X_{t-1},X_{t})$  of ${\bf X}$ at time $t+1$.
Note that the past histories have space-time components.
For $F\subset I$ and $t\in\Z$, $X_{F,t}=(X_{i,t},i\in F)$ denotes the configuration of 
${\bf X}$ at time $t$ restricted to set $F$. 
More generally, for any subset $v\subset I\times \Z$,  $X_v$ denotes the collection $(X_{i,t})_{(i,t)\in v}$.

\begin{assumption}
\label{Mod_Assump:1}
For each $t\in\Z$, given the past history $X_{-\infty:t}$, the neurons  spike independently of each other at time $t+1$, i.e., 
for any finite set $J\subset I$ and  $(a_i)_{i\in J} \in \{0,1\}^J,$
\begin{equation*}
\P(\cap_{i\in J}\{X_{i,t+1}=a_i\}| X_{-\infty:t}=x)=\prod_{i\in J} \P(X_{i,t+1}=a_i |  X_{-\infty:t}=x) \  \ \P\mbox{-a.e.} \ x\in\{0,1\}^{I\times \Z_{-}}.
\end{equation*}
\end{assumption}
 
Since the stochastic chain ${\bf X}$ is stationary, Assumption \ref{Mod_Assump:1} implies that the dynamics of ${\bf X}$ is fully characterized by the transition probabilities
\begin{equation*}
p_i(x)=\P(X_{i,0}=1| X_{ -\infty:-1}=x),  \ x\in\{0,1\}^{I\times \Z_{-}}, \ i\in I.
\end{equation*}
These transition probabilities are all assumed to be measurable functions of $x\in\{0,1\}^{I\times \Z_{-}}.$

Hereafter, we need the following notation.  For any neighborhood $v\subset I\times \Z_{-}$ and $x,y\in \{0,1\}^{I\times\Z_{-}}$, we write $x \stackrel{v}{=} y$ to indicate  $y_v=x_v$. 
For any real-valued function $f$ on $\{0,1\}^{I\times\Z_{-}}$ and subset $v\subset I\times \Z_{-}$,  we say $f$ is {\it cylindrical} in $v$  and write $f(x)=f(x_v)$, if $f(x)=f(y)$ for any $x,y\in \{0,1\}^{I\times \Z_{-}}$ such that $x \stackrel{v}{=} y$. 

Let us now present briefly the three other main probabilistic assumptions that we may require  in the article. They will be discussed in more details in the sequel.

We denote by $\mathcal{V}$ the collection of finite neighborhoods, i.e. finite subsets of $I\times \Z_{-}$ and we consider processes for which the following decomposition holds. 

\begin{assumption}[Space-time decomposition]
\label{Mod_Assump:2}
For all $v$ in $\mathcal{V}$ and $i$ in $I$,  there exists a $[0,1]$-valued measurable function $p_i^v(.)$, cylindrical in $v$,  and a non negative weight $\lambda_i(v)$,  such that for all $x\in\{0,1\}^{I\times \Z_{-}}$ and $i\in I$, 
$$
\left \{
\begin{array}{l}
     p_i(x) = \lambda_i(\emptyset) p_i^\emptyset(x) + \sum_{v\in \cV, v\neq \emptyset} \lambda_i(v) p_i^{v}(x),\\
     ~\\
   \sum_{v\in \mathcal{V}} \lambda_i(v) = 1.\\
\end{array}
\right.$$
\end{assumption}

As we will see in Section \ref{mod} where we discuss in details the implication of this assumption, this Kalikow-like assumption is actually met by a wide variety of stochastic processes modeling neuronal networks.

For all neighborhood $v\in\mathcal{V}$, let  $T(v)= -\min(s \in \Z_{-}, (j,s) \in v)$ be the corresponding time length, with the convention that $T(\emptyset)=0$. We also assume the following.
\begin{assumption}
\label{ass:exp_branch_proc}
There exists a strictly positive $\theta$ such that for all $i$,
$$\varphi_i(\theta)=\sum_{v\in \cV}  |v| e^{\theta T(v)}\lambda_i(v)$$
is finite and 
\begin{equation}\label{condlaplace}
 \varphi(\theta)=\sup_{i\in I} \varphi_i(\theta) <1.
\end{equation} 
\end{assumption}
This assumption, which can be seen as {\it probabilistic sparsity} (see Section \ref{mod}), means that the neighborhoods used in the previous decomposition are in average nor too large neither too timely spread. 

Finally, we will use also the following assumption to control the Gram matrices in Section \ref{back}.
\begin{assumption}
\label{minmu}
There exists some positive $\mu$ such that for all $i\in I$, for all $x$,
$$\mu \leq p_i(x)\leq 1-\mu,$$
\end{assumption}
This means that whatever the past, there is always enough randomness in the system.

Note that all these 4 assumptions are very general and are there to control the randomness of the underlying neuronal system without specifying any  parametric models. In particular in what follows,  the $p_i$'s are just approximated by a finite combination of elements of dictionaries, without being one  (see Section \ref{lass}): the overall purpose is to approximate $p_i$ by what happens in a finite subset of observed neurons, namely $F$, always keeping in mind that $F$ is embedded in a much more complex and potentially infinite network $I$, the overall complex stochastic interactions inside $I$ being governed by Assumptions \ref{Mod_Assump:1},  \ref{Mod_Assump:2}, \ref{ass:exp_branch_proc} and/or \ref{minmu} depending on the results.

\section{Lasso method and statistical notation\label{lass}}

For a finite $F\subset I$, subset of observed neurons, and integers $T>m\geq 1$ measuring the observation window, the aim is to estimate $x \mapsto p_i(x)$ for a fixed neuron $i\in F$, given the sample $X_{F,-(m-1):T}$.
To that end, for  each time $1\leq t\leq T$, we compare the past $X_{F,(t-m):(t-1)}$ to the current observation $X_{i,t}$ to guess what can  be a good approximation of $p_i(x)$.  
The intuition behind this strategy is that for a well-chosen  space-time neighborhood $v\subset I\times \Z_{-}$, it might be sufficient to know $x_v$ and not the whole past configuration $x$ to well approximate $p_i(x)$.


Given the sample $X_{F,-(m-1):T}$, one might consider several candidates to approximate $p_i(x)$. Here, we shall approximate $p_i(x)$ by linear combinations of a given dictionary  $\Phi$, i.e. a finite set of real-valued functions on $\{0,1\}^{I\times \Z_{-}}$ which are cylindrical in
$F\times \underline{m}$ with $\underline{m}=\{-m,\ldots, -1\}.$ More precisely, for each vector $a=(a_{\varphi})_{\varphi\in\Phi}\in \R^{\Phi}$, we denote
\begin{equation}
x\mapsto  f_a(x)=\sum_{\varphi\in\Phi}a_{\varphi}\varphi(x),
\end{equation}
the candidate encoded by the vector $a$ that should approximate $p_i(x)$. 
We assume that the functions in the dictionary are bounded in infinite norm by $\|\Phi\|_\infty$.

 We use the least-square contrast defined by
$$
C(f_a)=-\frac{2}{T}\sum_{t=1}^{T} f_a(X_{F,(t-m):(t-1)})X_{i,t}+\frac{1}{T}\sum_{t=1}^{T}f^2_a(X_{F,(t-m):(t-1)}), \ a=(a_{\varphi})_{\varphi\in\Phi}\in \R^{\Phi}.
$$
Observe that if for real-valued functions $f$ and $g$ on $\{0,1\}^{I\times \Z_-},$ we denote $\langle f, g \rangle_{T} = \frac{1}{T}\sum_{t=0}^{T-1} f(X_{-\infty:t})g(X_{-\infty:t})$ and $\|f\|_T$ the corresponding norm, then one has
$$
C(f_a)=-2\langle f_a, X_{i,\cdot+1} \rangle_{T} + \|f_a\|^2_T=\|f_a-X_{i,\cdot+1}\|^2_T- \|X_{i,\cdot+1}\|^2_T, 
$$
and  $C$ is minimum when $\|f_a-X_{i,.+1}\|^2_T$ is minimum. In this sense, minimizing $C$ over functions that are only depending on the past might give a good estimator of $p_i(x)$. 

Notice also that, if for $\varphi, \varphi'\in\Phi$ we write, 
\begin{equation}
\label{def:b_and_G}
b_{\varphi}=\frac{1}{T}\sum_{t=1}^{T}\varphi(X_{F,(t-m):(t-1)})X_{i,t} \quad \mbox{and} \quad \ G_{\varphi, \varphi'}=\frac{1}{T}\sum_{t=1}^{T}\varphi(X_{F,(t-m):(t-1)})\varphi'(X_{F,(t-m):(t-1)}),
\end{equation}
then $C(f_a)$ can be rewritten as
$$
-2 a^\intercal b+ a^\intercal G a,
$$
where $b=(b_{\varphi},\varphi\in\Phi)$ is a vector of $\R^\Phi$, $G=(G_{\varphi,\varphi'})_{\varphi, \varphi'\in\Phi}$ is the Gram matrix and $a^\intercal$ is the transpose of vector $a$. 

In the sequel, let  $|a|=(|a_{\varphi}|,\varphi\in\Phi)$, $|a|_{\infty}=\max_{\varphi\in\Phi}|a_{\varphi}|$, $\|a\|=\sqrt{ a^\intercal a}$ and $|a|_1={\bf 1}^\intercal |a|$ where ${\bf 1}$ is the vector with all coordinates equal to $1$.

Following \cite{hrbr}, we minimize $C(f_a)$ subject to a $\ell_1$-penalization on the vector $a=(a_{\varphi},\varphi\in\Phi)$. Precisely, we choose the function $\hat{f}=f_{\hat{a}}$ where 
\begin{equation}\label{lasso}
\hat{a}\in \arg\min_{a\in\R^{\Phi}}\left\{-2 a^\intercal b+ a^\intercal G a+ \gamma d |a| \right\},
\end{equation}
for $d$ a positive term controlling the random fluctuations  and $\gamma>0$, a tuning parameter.

The active set $S(a)$ of a vector $a\in\R^{\Phi}$ is the set $S(a)=\{\varphi: a_{\varphi}\neq 0 \}.$
We shall denote for any subset $J\subset \Phi$ and any $a \in \R^\Phi$, $a_J\in \R^\Phi $ the vector whose coordinates  in $J$ are equal to the ones of $a$ and $0$ anywhere else. We also denote by $|J|$ the cardinality of $J$.

For later use,  let us write for each $\varphi\in\Phi,$
\begin{equation}
\label{def:barb}
\bar{b}_{\varphi}=\frac{1}{T}\sum_{t=0}^{T-1}\varphi(X_{F,t-m:t})p_i(X_{-\infty:t}).
\end{equation}

\subsection{Examples of dictionaries}
\label{Ex_dict}
 Let us present briefly some examples of dictionaries that might be useful.

\begin{paragraph}{Short memory effect}
Let the dictionary  $\Phi$ be defined by the set $\{\varphi_j: j\in F\}$ where
$$\varphi_j(x)=
\begin{cases}
1, \ \mbox{if} \ x_{j,s}=1 \ \mbox{for some} \ -m\leq s \leq -1\\
0, \ \mbox{otherwise}
\end{cases}, \ x\in\{0,1\}^{I\times \Z_-},$$
so that we are trying to explain the presence of a spike on neuron $i$ at time $t$ by a linear combination of the presence of a spike on neuron $j$ in a small window just before time $t$.
\end{paragraph}

\begin{paragraph}{Cumulative effect}
We can also think that $m=\eta L$ is a much larger parameter and cut the past $\underline{m}$ into $L$ small pieces of length $\eta$, where the effect of the spikes are different and  cumulative. This leads to  the dictionary  $\Phi$ defined by the set $\{\varphi_{j,\ell}: j\in F \ \mbox{and} \ 1\leq \ell\leq L \}$ where
$$\varphi_{j,\ell}(x) =\sum_{s=-\eta \ell}^{-\eta(\ell-1)-1} x_{j,s}, \ x\in\{0,1\}^{I\times\Z_-}.$$
\end{paragraph}

\begin{paragraph}{Cumulative effect with spontaneous apparition}
It can be important to take into account a  background activity, especially to explain the apparition of spikes due to the unobserved part of the network. To do so, we may add to the previous dictionary an extra function
$$\varphi_0=1,$$
whose corresponding coefficient corresponds to a spontaneous activity.

\end{paragraph}

\begin{paragraph}{Hawkes dictionary}
In both the cumulative effect and the cumulative effect with spontaneous part, one might be interested in a particular example where $\eta=1$ and $L=m$. In particular, in the case with spontaneous part, we are therefore  interested in approximating 
$p_i(x)$ by
$$f_a(x)= a_0 + \sum_{j\in F}\sum_{-m\leq s\leq -1}a_{j,-s} x_{j,s}, \ x\in\{0,1\}^{I\times\Z_-},$$
which is the exact form of a discrete Hawkes process restricted to $F\times \underline{m}$ (see Section \ref{mod}) and this even if $p_i$ is not of this shape.
\end{paragraph}

Note that whatever the dictionary, $m$ represents the maximal delay in the dictionary and $|F|$ the number of observed neurons. As we will see in Section \ref{back}, both these quantities  have to be usually less than a certain increasing function of $T$, which depends on the dictionary (typically $\log(T)$), to derive an RE property on the Gram matrix. In particular $|m|$ might grow slightly with $T$ to ensure a good asymptotic approximation of the dependency in time.  Mathematically speaking, the same holds for $|F|$, although the size of $F$ is dictated by the neurophysiological experiment and, for practical prupose, it is always thought to be a constant with respect to $T$.

\subsection{Oracle inequality}

It is classical, by now, to derive oracle inequalities for Lasso procedures if $G$ satisfies some properties. We use two of them.

\begin{definition}
\label{Ass:Gram_matrix_invert}
Let $\kappa>0$. The matrix $G$ satisfies Property ${\bf Inv}(\kappa)$ if 
$$\forall a\in \R^{\Phi}, \quad a^\intercal G a \geq \kappa \|a\|^2$$ 
\end{definition}

A weaker version  is the restricted eigenvalue condition. 
\begin{definition}
\label{Ass:RE}
Let $c>0$, $\kappa>0$ and $s\in \mathbb{N}$. The matrix $G$ satisfies Property ${\bf RE}(\kappa,c,s)$ if for all  subset $J$ such that $|J|\leq s$ and 
for all $a\in \R^{\Phi}$ such that $$ |a_{J^c}|_1\leq c |a_J|_1, $$
the following holds $$a^\intercal G a \geq \kappa \|a_J\|^2.$$
\end{definition}

Our first result establish an oracle inequality for the estimator $\hat{f}=f_{\hat{a}}$ where $\hat{a}$ is defined by \eqref{lasso}.

\begin{thm}
\label{thm:Oracle_inequality}
Let $\gamma\geq 2$, $\kappa>0$ and $s\in \mathbb{N}$.
On the event  on which
\begin{enumerate}
\item[(i)] for all $\varphi\in\Phi$, $|b_{\varphi}-\bar{b}_{\varphi}|\leq d$, \\
\item[(ii)] and $G$ satisfies ${\bf RE}(\kappa,c(\gamma),s)$ with $c(\gamma)=\frac{\gamma+2}{\gamma-2}$,
\end{enumerate}
the following inequality is satisfied
\begin{equation}\label{oracle}
\| \hat{f}-p_i(\cdot) \|^2_T\leq \inf_{a\in\R^{\Phi}: \ |S(a)|\leq s}\left\{\|f_a-p_i(\cdot) \|^2_T + \kappa^{-1} |S(a)| d^2\frac{(\gamma+2)^2}{4}\right\}.
\end{equation}
Moreover for any $0<\delta<1$, if $d=d_\delta$ with
$$d_\delta= \sqrt{\|\Phi\|^2_\infty \frac{\log|\Phi|+\log(2\delta^{-1})}{2T}},$$
then $\P(\exists \varphi\in\Phi: |b_{\varphi}-\bar{b}_{\varphi}|>d)\leq \delta.$
\end{thm}

Equation \eqref{oracle} is a classical oracle inequality (see for instance \cite{hrbr} or \cite{HRBRSW:18} for close set-ups). This result means that the Lasso estimator  gives the best $s$-sparse approximation of $p_i$ based on the dictionary $\Phi$ and that the price to pay is of the order of  $\kappa^{-1} s d^2$, if we assume that $\|\Phi\|_\infty \leq 1$. With the choice $d=d_\delta$,  we have therefore   a price of the order of $\kappa^{-1} s \frac{\log|\Phi|+\log(2\delta^{-1})}{T}$. Note that if we knew that $p_i$ can be indeed decomposed on $\Phi$, meaning that the model is true and that in particular $p_i$ only depends on $s$ elements of the dictionary $\Phi$, the price to pay anyway to estimate $p_i$  would be roughly of the order of $s /T$. Therefore if the logarithmic factor is a classical loss for adaptation in \eqref{oracle}, it remains to see the order of  $\kappa$, to see if \eqref{oracle} gives roughly the best possible rate.

Note that if $G$ satisfies ${\bf{Inv}}(\kappa)$ then one can choose $\gamma=2$ and $s=|\Phi|$ in Theorem \ref{thm:Oracle_inequality} and \eqref{oracle} can be rewritten as 
$$
\| \hat{f}-p_i(\cdot) \|^2_T\leq  \inf_{a\in\R^{\Phi}}\left\{\|f_a-p_i(\cdot) \|^2_T + 4\kappa^{-1} |S(a)| d^2 \right\},
$$
which is a sharper version of the result proved in \cite{hrbr} in continuous time, up to the fact that
they used more general weights which leads to a weighted $\ell_1$ norm in the criterion. The same refinement would have been possible but since the focus is here on the Gram matrix, we have decided to use a classical $\ell_1$ norm for sake of simplicity.

Note also that another (very easy) refinement consists in clipping $\hat{f}$ to ensure that it remains between 0 and 1. The same result holds for this clipped version. Another way to find similar results for estimators that are forced to be in $[0,1]$ is to use  penalized maximum likelihood. Many works  have used it (see for instance \cite{MRW:18} for Poisson counts  or \cite{BM:15, GM:19} in Gaussian Markovian set-ups). This comes with additional technicalities, in particular if the likelihood of the statistical model is not easy to compute, because the model is not Gaussian. In particular,  \cite{MRW:19} use  a setting very close to ours, but simpler (see also the examples in Section \ref{mod}) and make use of Taylor expansion to approximate the criteria. Translated here, the approximation would depend on the dictionary we use and would be more complex for each dictionary. Once again,  because the focus is here on the Gram matrix, we have decided to stick with the simplest Lasso result made for least-square contrast.

Results for controlling Gram matrices are numerous (see for instance \cite{BM:15,GM:19,HRBRSW:18} for simpler settings than the present one) but always assume that the whole network is observed and that {\it the target can be written on the dictionary}.
In \cite{hrbr}, which is the closest framework to the present one, it has 
been proved for instance that, if one observes the whole finite network and if the spike trains are linear Hawkes processes, then $G$ is invertible with large probability for well chosen dictionaries. In this case, the corresponding $\kappa$
is roughly lower bounded by a quantity which is exponentially small in the total number of neurons in the network. Here we would like to go beyond these assumptions and prove that even if 
\begin{itemize}
\item the model is wrong (i.e. $p_i$ is not Hawkes for instance) and $p_i$ cannot be written on the dictionary,
\item the network is infinite,
\item  we only observe a very partial subnetwork,
\end{itemize}
it is still possible to find good $\kappa$ with high probability and that the dependency in the number of neurons can be much better than these previous results.

The main idea consists in using very general Kalikow-type decomposition of the transition probabilty $p_i(x)$, that are available in discrete time  (see \cite{GGLO:13}) and that do not exist with such generality in continuous time (see however \cite{evapierre} for promising results in this direction).  

\section{Space-time decomposition and concentration\label{mod}}

\subsection{Definition}

Let us  first recall Assumption \ref{Mod_Assump:2}:
For all $v$ in $\mathcal{V}$ and $i$ in $I$,  there exists a $[0,1]$-valued measurable function $p_i^v(.)$, cylindrical in $v$,  and a non negative weight $\lambda_i(v)$,  such that for all $x\in\{0,1\}^{I\times \Z_{-}}$ and $i\in I$, 
$$
\left \{
\begin{array}{l}
     p_i(x) = \lambda_i(\emptyset) p_i^\emptyset(x) + \sum_{v\in \cV, v\neq \emptyset} \lambda_i(v) p_i^{v}(x),\\
     ~\\
   \sum_{v\in \mathcal{V}} \lambda_i(v) = 1.\\
\end{array}
\right.$$

The aforementioned decomposition can be interpreted as follows. At each time step, to decide whether the neuron $i$ spikes or not, we first choose a random space-time neighborhood $V\in\cV$ according to the distribution $\lambda_i$. Once this neighborhood $V$ is chosen, we decide if a spike of neuron $i$ occurs with probability $p_i^V(x_V)$ that depends only on the past history restricted to $V$. Note that $p_i^\emptyset(x)$ does not depend on $x$ at all, and we denote this value $p_i^\emptyset$. 

Such a space-time decomposition of the transition probabilities $\{p_i(x),i\in I,x\in\{0,1\}^{I\times \Z_{-}}\}$ generalizes the classical  Kalikow decomposition introduced in \cite{Kalikow:90} and further developed in \cite{CFF:02}, \cite{GGLO:13} and \cite{GL:13}.  The main difference consists in not forcing  the nesting of the neighborhoods $v$ that lie in the support of $\lambda_i$.  
This helps us to exploit the fact that in many cases the distributions $\lambda_i$ charge very few neighborhoods and that the cardinality of this neighborhood is usually very small, if the nesting is not forced. We speak in this case of {\it probabilistic sparsity}.

\begin{remark}
\label{rmk:1}
If we denote $q_i(x)=\P(X_{i,0}=0| X_{-\infty:-1}=x)=1-p_i(x)$ and $q^{v}_i(x)=1-p^{v}_i(x)$  for all $i\in I$, $x\in\{0,1\}^{I\times \Z_{-}}$, we can also write
$$
q_i(x)=\lambda_i(\emptyset) q_i^\emptyset(x) + \sum_{v\in \cV, v\neq \emptyset} \lambda_i(v) q_i^{v}(x),
$$ 
where for each  $v\in \mathcal{V}$, the function $q_i^v$ is cylindrical in $v$.
\end{remark}

\begin{remark}
For a given  space-time decomposition, one can use Remark \ref{rmk:1} to deduce that for all $i \in I,$
$$
\inf_{x\in\{0,1\}^{I\times \Z_{-}}}p_i(x)+\inf_{x\in\{0,1\}^{I\times \Z_{-}}}q_i(x)\geq \lambda_i(\emptyset).
$$
More generally, for any $v\in\cV$, one can show that
$$
\inf_{x\in \{0,1\}^{I\times \Z_{-}}}\left\{ \inf_{y\in\{0,1\}^{I\times \Z_{-}}:\ y \ \stackrel{v}{=} \ x}p_i(y)+\inf_{y\in\{0,1\}^{I\times \Z_{-}}:\ y \ \stackrel{v}{=} \ x}q_i(y)\right\}\geq \lambda_i(\emptyset)+\sum_{w\subseteq v, w\neq\emptyset}\lambda_i(w).
$$

One can also show that the  space-time decomposition is not unique.
This fact raises the question of whether there is an ``optimal'' decomposition of a given  transition probability. Such a question, however, is not discussed in this article. 
\end{remark}

\subsection{Main examples}
\label{Sec:main_examples}

\subsubsection{Markov chains}
Suppose $I$ is a singleton, say $I=\{1\}$, and denote $X_t$ instead of $X_{1,t}$ for convenience. Let us assume also that for all $x\in\{0,1\}^{\Z_{-}},$ 
$$
p(x)=\P(X_{0}=1| X_{-\infty:-1}=x)=\P(X_{0}=1| X_{-1}=x_{-1}),
$$
that is, the stochastic chain ${\bf X}=(X_t)_{t\in \mathbb{Z}}$ is a Markov chain of order 1 taking values in $\{0,1\}$. To shorten notation,  let 
$$
p^1=\P(X_t=1|X_{t-1}=1), q^1=1-p^1,
p^0=\P(X_t=1|X_{t-1}=0) \ \mbox{and} \ q^0=1-p^0.
$$
We can always write the transition probability $p(x)$ as 
$$p(x)=p^1x_{-1}+p^0(1-x_{-1}).$$
Let us denote $p=p^1\wedge p^0$, $q=q^1\wedge q^0$ and $\mu=p+q$.
If $0<\mu<1$, then  one can write
$$p(x)= \mu\frac{p}{\mu} + (1-\mu) \left[ \frac{p^1-p}{1-\mu}x_{-1}+ \frac{p^0-p}{1-\mu}(1-x_{-1})\right].$$
So one can use as space-time decomposition 
\begin{equation}
\left\{\begin{array}{lcll}
\lambda(\emptyset) &= & \mu \\
 p^\emptyset &= & \frac{p}{\mu}\\
\lambda(\{-1\}) &=&1-\mu\\
p^{\{-1\}} (x) &=&  \frac{p^1-p}{1-\mu}x_{-1}+ \frac{p^0-p}{1-\mu}(1-x_{-1})\\
\end{array}
\right..
\end{equation}
Note that the support of $\lambda$ is then reduced to $\{\emptyset,\{-1\}\}$.
Note also than it is quite straightforward to extend this to the multidimensional models considered in \cite{MRW:19}, for instance by saying that when the neighborhood is not empty, one picks $I$ which is finite in their case.

\subsubsection{Chains of infinite order}
Again suppose $I$ is a singleton. In this case, the stochastic chain ${\bf X}$ is described by the transition probability $\{p(x),x\in\{0,1\}^{\Z_{-}}\}$.  Denote for $ \ell\in\Z_+$, $\underline{\ell}$ the set $\{-\ell,...,-1\}$ and 
\begin{equation*}
\beta_{\ell}=\sup_{x\in \{0,1\}^{\Z_{-}}}\sup_{\begin{array}{c}y,z \in \{0,1\}^{\Z_{-}} \mbox{ s.t.} \\ y\stackrel{\underline{\ell}}{=}z\stackrel{\underline{\ell}}{=}x\end{array}}
\{|p(y)-p(z)| \}.
\end{equation*}
If there exist $\ell_0\geq 1$ such that $\beta_{\ell}=0$ for all $\ell\geq \ell_0$, then the stochastic chain ${\bf X}$ is called Markov Chain of Order $\ell_0$.
Otherwise, ${\bf X}$ is called Chain of Infinite Order. We refer the reader to \cite{FFG:01} for a comprehensive introduction to Chains of Infinite Order.

If $\beta_{\ell}\to 0$ as $\ell\to\infty$, then $\{p(x),x\in\{0,1\}^{\Z_{-}}\}$ is said to be {\it continuous} and the sequence $(\beta_{\ell})_{\ell\in\Z_+}$ is called the {\it continuity rate}.
Recall that $q(x)=1-p(x)$ for all $x\in\{0,1\}^{\Z_{-}}$. 
One can then compute  for $\ell\in\Z_+$,
$$
\alpha(\ell)=\inf_{x\in \{0,1\}^{\Z_{-}}} \left\{\inf_{y\in \{0,1\}^{\Z_{-}} \ s.t.\ y\stackrel{\underline{\ell}}{=}x}p(y)+\inf_{y\in \{0,1\}^{\Z_{-}} \ s.t.\   y\stackrel{\underline{\ell}}{=}x}q(y)\right\}.
$$
This allows us to define the distribution $\lambda$ which has support only on the $\underline{\ell}$'s and
\begin{equation}
\label{def:coef_kali_decomp_inf_order}
\lambda( \underline{\ell})=\alpha(\ell)-\alpha(\ell-1),
\end{equation}
where $\alpha(0)=\lambda(\emptyset)=\inf_{x\in \{0,1\}^{\Z_{-}}}p(x)+\inf_{x\in \{0,1\}^{\Z_{-}}}q(x)$. One can show that every continuous transition probability $\{p(x),x\in\{0,1\}^{\Z_{-}}\}$ admits a decomposition of the form:
\begin{equation}
\label{def:kali_decom_inf_order}
\begin{cases}
p(x)=\lambda(\emptyset)p^{\emptyset}+\sum_{\ell\in\Z_+}\lambda(\underline{\ell})p^{\underline{\ell}}(x)\\
\lambda(\emptyset)+\sum_{\ell\in\Z_+}\lambda(\underline{\ell})=1
\end{cases}.
\end{equation}
Moreover  \eqref{def:kali_decom_inf_order} is a space-time decomposition since  $p^{\emptyset}\in [0,1]$  and for each $\ell\in\Z_+,$  
$\{p^{\underline{\ell}}(x),x\in\{0,1\}^{\Z_{-}}\}$ is a transition probability of a Markov chain of order $\ell$. 

\subsubsection{Discrete-time linear Hawkes processes}
\label{Ex:Hawkes_processes}
Multivariate Hawkes processes (also referred in the neuroscience literature as a particular case of generalized linear models)  are often used to model interacting spike trains and especially the synaptic integration. In contrast to the classical framework where these processes are described in continuous time and are not linear, we focus here on a discrete-time and linear formulation,  where for $i\in I:$
\begin{equation}
\label{def_Hawkes:1}
\psi_i(x)=\nu_i+\sum_{s\in\Z_{-}}\sum_{j\in I}h_{j\to i}(-s)x_{j,s} , \mbox{ and } 
\begin{cases}
p_i(x)= \psi_i(x),  \ \mbox{ if } \psi_i(x) \in [0,1],\\
p_i(x)= 1,   \mbox{ if } \psi_i(x) >1, \\
p_i(x)=0,  \ \mbox{ if } \psi_i(x) <0. \\
\end{cases}
\end{equation}
In this formula,
 $\nu_i\geq 0$ represents the spontaneous activity of neuron $i$, that is its ability to produce spikes when there is no interaction. The interaction function $h_{j\to i}$ measures the amount of excitation (if positive) or inhibition (if negative) that a spike of neuron $j$  has on neuron $i$ after a delay $-s$ (a spike of neuron $j$ with delay $-s$ corresponds to $x_{j,s}=1$). 
  
For a given neuron $i\in I$, we write
$$A_i^+=\{(j,s)\in I\times \Z_{-}: h_{j\to i}(-s)>0\} \mbox{ and } A_i^-=\{(j,s)\in I\times \Z_{-}: h_{j\to i}(-s)<0\},$$
and define the maximal excitatory (respectively inhibitory) strength by 
$$\Sigma_i^+=\sum_{(j,s)\in A_i^+} |h_{j\to i}(-s)|  \mbox{ and }\Sigma_i^-=\sum_{(j,s)\in A_i^-} |h_{j\to i}(-s)|.$$
Let us assume that 
\begin{equation}
\label{Ass_Hawkes}
0\leq \nu_i-\Sigma_i^- \quad\mbox{ and } \quad  \nu_i+\Sigma_i^+ \leq 1,
\end{equation}
which implies in particular that whatever the past configuration $x\in\{0,1\}^{I\times \Z_{-}}$, the transition probability $p_i(x)\in [0,1]$ is always equal to $\psi_i(x)$. It also implies  that $\Sigma_i^+ +\Sigma_i^- \in [0,1]$. 

Then one can use for the space-time decomposition:
\begin{equation}
\label{Kalikow_decomp_HP}
\left\{\begin{array}{lcll}
\lambda_i(\emptyset) &= & 1-(\Sigma_i^++\Sigma_i^-) &\mbox{ which is } \geq 0 \mbox{ since } 0\leq \Sigma_i^++ \Sigma_i^-\leq 1, \\
 p_i^\emptyset &= & \frac{\nu_i-\Sigma_i^-}{\lambda_i(\emptyset)} &\mbox{ which is } \leq  1  \mbox{ since } \nu_i+\Sigma_i^+ \leq 1,\\
\lambda_i(\{(j,s)\}) &=&|h_{j\to i}(-s)| &\mbox{ for all } (j,s) \in A_i^+\cup A_i^-,\\
p_i^{\{(j,s)\}} (x) &=&   x_{j,s} &\mbox{ for all } (j,s) \in A_i^+,\\
p_i^{\{(j,s)\}} (x) &=&   (1-x_{j,s}) &\mbox{ for all } (j,s) \in A_i^-.
\end{array}
\right.
\end{equation}
It is moreover sufficient to assume that $  \Sigma_i^++ \Sigma_i^-< 1$ to have $\lambda_i(\emptyset)>0$.

The discrete-time linear Hawkes model is an interesting example, because even if the true interaction graph, that is the set of edges $(j,i)\in I\times I $ for which $h_{j\to i}$ is non zero,  is complete,  the random neighborhoods $V\in\cV$ of the space-time decomposition have cardinality at most 1 almost surely. This {\it probabilistic sparsity} is exploited in the sequel to obtain concentration inequalities.

Note that it is classically assumed for general Hawkes models that the spectral radius of the matrix $(\int |h_{j\to i}|)_{i,j}$ is smaller than 1 to ensure stationnarity of the whole multivariate process. Here this matrix can be reinterpreted as $H=(\sum_{\ell>1} |h_{j\to i}(\ell)|)_{i,j}$. Therefore  \eqref{Ass_Hawkes} is different from the usual assumption: it implies in particular that 
$$\Sigma^++\Sigma^-=H\bf{1} \leq \bf{1} \mbox{ coordinate per coordinate}$$ 
where $\Sigma^+=(\Sigma_i^+)_i$, $\Sigma^+=(\Sigma_i^+)_i$ and ${\bf 1}$ is the vector of 1's. 

\subsubsection{GL neuron model}
Let $W_{j \to i}\in \R$ with $i,j\in I$,  be a collection of real numbers such that $W_{j \to j } = 0 $ for all $j.$ For each $i\in I$, let $\varphi_i : \R  \to [ 0, 1 ]$ be a  non-decreasing measurable function and $g_i=(g_i(\ell))_{\ell\in\Z_+}$ be a sequence of strictly positive real numbers. 
Here, $ W_{j \to i } $ is interpreted as the {\it synaptic weight of neuron $j$ on neuron $i,$} $\varphi_i $ as the  {\it spike rate function} of neuron $i$ and $g_i $ as the {\it postsynaptic current pulse} of neuron i.

For each $x\in\{0,1\}^{I\times\Z_{-}}$ and $i\in I$, we define $L^i(x)=\sup\{s\in\Z_-: x_{i,s}=1\}.$ The stochastic chain ${\bf X}$ satisfies a  GL neuron model if the transition probabilities $\{p_i(x),i\in I, x\in\{0,1\}^{I\times\Z_{-}}\}$ are given by (see \cite{GL:13})
\begin{equation}
\label{def:trans_prop_GL}
p_i(x)
=
\begin{cases}
\varphi_i(0), \mbox{if} \ L^i(x)=-1,\\
\varphi_i\left(\sum_{j\in I} W_{j\to i} \sum_{s=L^i(x)+1}^{-1}g_j(-s)x_{j,s}\right), \mbox{otherwise}. 
\end{cases}
\end{equation}
\begin{remark}
Notice that the functions $L^i$'s introduce a structure of variable-length memory in the model. For this reason the GL neuron model was introduced in \cite{GL:13} under the name of {\it Systems of Interacting Chains with Memory of Variable Length}.
\end{remark}

\begin{paragraph}{Linear spike rate functions}

Let us consider the particular case where the parameters of the model are such that $\varphi_i(u)=\nu_i+u$ with $\nu_i\geq 0$ for each $i\in I$. 
Similarly to Section \ref{Ex:Hawkes_processes}, let us denote for each $i\in I,$ 
$$A_i^+=\{(j,s)\in I\times\Z_{-}: W_{j\to i}g_j(-s)>0\} \mbox{ and } A_i^-=\{(j,s)\in I\times\Z_{-}: W_{j\to i}g_j(-s)<0\},$$
and define the maximal excitatory (respectively inhibitory) strength by 
$$\Sigma_i^+=\sum_{(j,s)\in A_i^+} |W_{j\to i}|g_j(-s)  \mbox{ and }\Sigma_i^-=\sum_{(j,s)\in A_i^-} |W_{j\to i}|g_j(-s).$$
We also assume that 
\begin{equation}
\label{Ass_Linear_GL}
0\leq \nu_i-\Sigma_i^- \quad\mbox{ and } \quad  \nu_i+\Sigma_i^+ \leq 1.
\end{equation}
Under these assumptions, one can check that the transition probabilities \eqref{def:trans_prop_GL} also satisfy Assumption \ref{Mod_Assump:2}. 
Specifically, we can use
\begin{equation}
\label{ST_kalikow_decom_linear_GL}
\left\{\begin{array}{lcll}
\lambda_i(\emptyset) &= & 1-(\Sigma_i^++\Sigma_i^-) &\mbox{ which is } \geq 0 \mbox{ since } 0\leq \Sigma_i^++ \Sigma_i^-\leq 1, \\
 p_i^\emptyset &= & \frac{\nu_i-\Sigma_i^-}{\lambda_i(\emptyset)} &\mbox{ which is } \leq  1  \mbox{ since } \nu_i+\Sigma_i^+ \leq 1,\\
\lambda_i(\{(j,s)\}^{\downarrow i}) &=&|W_{j\to i}|g_j(-s) &\mbox{ for all }  (j,s) \in A_i^+\cup A_i^-,\\
p_i^{\{(j,s)\}^{\downarrow i}} (x) &=&   x_{j,s} 1_{x_{i,s:-1}=0} &\mbox{ for all } (j,s) \in A_i^+,\\
p_i^{\{(j,s)\}^{\downarrow i}} (x) &=&   (1-x_{j,s})1_{x_{i,s:-1}=0} &\mbox{ for all } (j,s) \in A_i^-,
\end{array}
\right.
\end{equation}
where $\{(j,s)\}^{\downarrow i}=\{(j,s),(i,s),\ldots ,(i,-1)\}$ is the augmentation of the set $\{(j,s)\}$ on the coordinate $i$ for each $(j,s)\in A_i^+\cup A_i^-.$ 
Hence, the random neighborhoods $V\in\cV$ have cardinality either $0$ (when $V=\emptyset$) or $s+1$ (when $V=\{(j,s)\}^{\downarrow i}$ with $j\neq i$) or $s$  (when $V=\{(i,s)\}^{\downarrow i}$ ).

\end{paragraph}

\begin{paragraph}{Non-linear spike rate functions}

In the previous work of \cite{GL:13}, the space-time decomposition is restricted to growing sequences of neighborhoods $v$ that are indexed by their range in time. For each $i\in I$, one assumes that there exists a growing sequence $J_i(1)=\{i\}, J_i(\ell) \subset J_i(\ell+1)$ of subsets of $I$ that corresponds to the space positions that are needed when looking at a past of length $\ell$, so that we can form $v_i(\ell)=J_i(\ell)\times\underline{\ell}$, defining a growing sequence of subsets of $I\times\Z_{-}$.

Next let us introduce the following quantities:
\begin{eqnarray*}
\label{def:kali_coeff_GL}
&&\alpha_i(\ell)=\inf_{x\in\{0,1\}^{I\times\Z_{-}}}\left\{\inf_{y\in\{0,1\}^{I\times\Z_{-}}:y\stackrel{v_i(\ell)}{=} x }p_i(y)+\inf_{y\in\{0,1\}^{I\times\Z_{-}}:y\stackrel{v_i(\ell)}{=} x }q_i(y)\right\} 
\end{eqnarray*}
and $\lambda_i(v_i(\ell))=\alpha_i(\ell)-\alpha_i(\ell-1)$, where for each $i\in I,$ $q_i(y)=1-p_i(y)$ and  $\lambda_i(\emptyset)=\alpha_{i}(0)=\inf_{x\in \{0,1\}^{I\times\Z_{-}}}p_i(x)+\inf_{x\in \{0,1\}^{I\times\Z_{-}}}q_i(x).$

Let us assume that
\begin{equation}
\label{Ass:kalikow_net_neurons}
\sup_{i\in I}\sum_{j\in I}|W_{j\to i}|<\infty, \ \sum_{\ell\in\Z_+}\sup_{i\in I}g_i(\ell)<\infty \ \mbox{and} \  \sup_{i\in I}|\varphi_i(u)-\varphi_i(v)|\leq \gamma |u-v|,  
\end{equation}
where $\gamma$ is a positive constant.

It has been proved in \cite{GL:13} (see Proposition 2) that the
transition probabilities $\{p_i(x),x\in\{0,1\}^{I\times\Z_{-}}\}$ admit the following  space-time decomposition:
\begin{equation}
\begin{cases}
p_i(x)=\lambda_i(\emptyset)p^{\emptyset}_i+\sum_{\ell\in \Z_+}\lambda_i(v_i(\ell))p^{v_i(\ell)}_i(x),\\
\lambda_i(\emptyset)+\sum_{\ell=1}^{+\infty}\lambda_i(v_i(\ell))=1,
\end{cases}
\end{equation}
with, $p^{\emptyset}_i\in [0,1]$ and for $\ell\geq 1$,  $p^{v_i(\ell)}_i(x)$ is a $[0,1]$-valued measurable function which is cylindrical in $v_i(\ell)$.

Hence, the transition probabilities $p_i$'s also satisfy  Assumption \ref{Mod_Assump:2} in the nonlinear case. The random neighborhoods $V\in\cV$ have cardinality either $0$ (when $V=\emptyset$) or $\ell |J_i(\ell)|$ (when $V=v_i(\ell)$). Note that in the non-linear case the neighborhoods $v_i(\ell)$ are dense in time by construction, whereas in the linear case one can obtain a stronger {\it probabilistic sparsity.}
\end{paragraph}

\subsection{Main properties}
\label{Sec:Perf_Simulation_and_conseq}

Before being able to assess a value to $X_{i,t}$ at site $(i,t)$ for fixed neuron  $i$ and time $t$, we need to understand  on which previous sites this value depends. To do so, we use the distribution $\lambda_i$ to obtain a space-time neighborhood of $(i,t)$.  More precisely, because the distribution $\lambda_i$ gives a neighborhood for neuron $i$ at time 0,  we need to shift it at time $t$ to obtain a realization of the random neighborhood for neuron $i$ at time $t$ by stationarity.
Hence if for every $t\in\Z$ and subset $A$ of $I\times \Z$,
$${A}^{\to t}=\{(j,s+t) \mbox{ for } (j,s)\in A\},$$
with the convention that $\emptyset^{\to t}=\emptyset,$ we can define the random neighborhood $K_{i,t}$ of site $(i,t)$ as $$K_{i,t}=V_{i,t}^{\to t}$$ where $V_{i,t}$ is drawn independently of anything else according to $\lambda_i$. We can proceed independently for all sites $(j,s)$ and obtain $K_{j,s}=V_{j,s}^{\to s}$. 

By looking recursively at the neighborhoods of the neighborhoods, we are building a whole genealogy in space and time of the site $(i,t)$, that is the list of sites that are really impacting the value $X_{i,t}$. This genealogy is random and depends only on the realizations of the neighborhoods, i.e. only on the distributions $\lambda_i$'s.

The study of this space-time genealogy is of utmost importance. Indeed if the genealogy is almost surely finite then we are able to follow classical constructions as done by \cite{GL:13} to write a perfect simulation algorithm. Moreover the study of the length of the genealogy enables us to cut time into almost independent blocks and therefore to have access to concentration inequalities, this second construction being inspired by \cite{V:97},  \cite{RBR:07} or \cite{hrbr}.

\subsubsection{Sufficient condition for finite genealogies}
 For all sites $(i,t)$, let us define recursively $A^1_{i,t}=K_{i,t}$ and for $n\geq 1,$
 $$A^{n+1}_{i,t}=(\cup_{(j,s)\in A^n_{i,t}}K_{j,s})\setminus\{A^1_{i,t}\cup\ldots\cup A^n_{i,t} \},$$ 
the genealogy stopped after $n+1$ generations. 

The complete genealogy is $G_{i,t}=\cup_{n=1}^\infty A^n_{i,t}$. It is finite if and only if  
$$N_{i,t}=\inf\{n\geq 1: A^{n}_{i,t}=\emptyset\},$$
is finite. 

This is a consequence of the following property.
\begin{assumption}
\label{ass:branch_proc_death}
For each $i\in I$, we assume that the mean size of the random neighborhood on neuron $i$ 
\begin{equation}
\label{def:offspring_mean}
\bar{m}_i=\sum_{v\in\cV}|v|\lambda_i(v),
\end{equation}
is finite and that the maximal mean size satisfies
\begin{equation}
\label{ass:bar_m_small_1}
\bar{m}=\sup_{i\in I}\bar{m}_i<1.
\end{equation}
\end{assumption}
 {\it Probabilistic sparsity} corresponds here to the fact that the mean size of the random neighborhoods are strictly less than 1.

Thanks to this assumption, we can prove the following result.
\begin{proposition}
\label{N_i_t_finite}
For each $i\in I$, $t\in\Z$  and $\ell\in\Z_+,$
\begin{equation*}
\P\left(N_{i,t}>\ell\right)\leq (\bar{m})^{\ell}.
\end{equation*}
In particular, under Assumption \ref{ass:branch_proc_death}, for all $i\in I$ and $t\in \Z$, 
\begin{equation}
\label{prob_N_i_t_finite}
\P(N_{i,t}<\infty)=1,
\end{equation}
that is all genealogies are finite almost surely.
\end{proposition}

\subsubsection{Perfect Simulation Algorithm}
Fix a site $(i,t)$ and suppose we want to simulate $X_{i,t}$. 

Under Assumption \ref{ass:branch_proc_death}, we know the genealogy is finite almost surely and it is possible to build this genealogy recursively without having to generate all the $V_{j,s}$. Once the genealogy is obtained by going backward in time, it  is then sufficient to go forward and simulate the $X_{j,s}$'s in the genealogy according to the transitions $p^{V_{j,s}}(X_{K_{j,s}}).$ 

More formally, we can use two independent fields of independent uniform random variables on $[0,1]$, ${\bf U^1}=(U^1_{i,t})_{i\in I, t\in\Z}$ and ${\bf U^2}=(U^2_{i,t})_{i\in I, t\in\Z}$, such that the whole randomness of the construction is encompassed in the field ${\bf U^1}$ for the genealogies and in the field ${\bf U^2}$ for the forward transitions and such that conditionally on these two fields, the whole simulation algorithm is deterministic. But in practice, we generate $U^1_{j,s}$ and $U^2_{j,s}$ only if we need it. This leads to the following algorithm

\begin{itemize}
\item[Step 1.] Generate $U^1_{i,t}$ random uniform variable on $[0,1]$. Since $\cV$ is countable, one can order its elements such that $\cV=\{v_1,...,v_n,...\}$. Define the c.d.f. of $\lambda_i$ by $F_i(0)=\lambda_i(\emptyset)$ and for $n\geq 1$,
$$
F_i(n)=\lambda_i(\emptyset)+\sum_{k=1}^n\lambda_i(V_{k})
$$
and pick the random neighborhood of $(i,t)$ as
$$
K_{i,t}=V_{i,t}^{\to t} \mbox{ with } V_{i,t}=
\begin{cases}
\emptyset, \ \mbox{if} \ U^1_{i,t}\leq F_i(0),\\
v_n, \ \mbox{if} \ F_i(n-1) < U^1_{i,t}\leq F_i(n)\ \mbox{for some} \ n\geq 1  
\end{cases}.
$$
Initialize $A^1_{i,t}\leftarrow K_{i,t}$.
\item[Step 2.]  Generate recursively $U^1_{j,s}$ for ${j,s}\in A^n_{i,t}$, compute the corresponding $V_{j,s}$ and $K_{j,s}$ as in Step 1 and actualize
$A^{n+1}_{i,t}\leftarrow \left(\cup_{{j,s}\in A^n_{i,t}}K_{j,s}\right) \setminus\{A^1_{i,t}\cup\ldots\cup A^n_{i,t} \} $.
 After a finite number of steps, $A^n_{i,t}$ is empty and [Step 2.] stops. Let $N_{i,t}$ be the final $n$ of this recursive procedure and  the genealogy of $(i,t)$ is given by  $G_{i,t}= \cup_{n=1}^{N_{i,t}}A^n_{i,t}. $
\item[Step 3.] Note that the $(j,s)$'s in $A^{N_{i,t}-1}_{i,t}$  have therefore an empty neighborhood.
Generate i.i.d. uniform variables $U^2_{j,s}$ for $(j,s)$ in  $A^{N_{i,t}-1}_{i,t}$ and define
\begin{equation}
\label{def:Simulation_algorithm_step_2_first_part}
X_{j,s}=1\{U^2_{j,s}\leq p^{\emptyset}_j\}.
\end{equation}
\item[Step 4.] Recursively generate $U^2_{j,s}$ for $(j,s)$ in $A^{\ell}_{i,t}$ recursively from $\ell=N_{i,t}-2$ to $ell=1$ and define
\begin{equation}
\label{def:Simulation_algorithm_step_2_recursively_part}
X_{j,s}=1\{U^2_{j,s}\leq p^{V_{j,s}}_j(X_{K_{j,s}})\},
\end{equation}
In particular arrived at $\ell=1$, one generates
\begin{equation}
\label{def:Simulation_algorithm_step_2_last_part}
X_{i,t}=1\{U^2_{i,t}\leq p^{V_{i,t}}_j(X_{K_{i,t}})\}.
\end{equation}
\end{itemize}

It is well-known that the algorithm above not only shows the existence but also the uniqueness of the stochastic chain {\bf X} compatible with Assumptions 
\ref{Mod_Assump:1}, \ref{Mod_Assump:2} and \ref{ass:branch_proc_death}
(see for instance \cite{GL:13} for formal statement of this result in a close setup). 

Note that when simulating the linear Hawkes process, the algorithm reduces to a random walk in the past to find the genealogy, a random decision on the state $X_{j,s}$ at the end of the random walk and a forward decision of the other states $X_{j,s}$ which is then completely deterministic and just depends on the sign of $h_{j\to i}(-s)$.

\subsubsection{Time length of a genealogy}
We are now interested by the time length of a genealogy. Let, for each non-empty subset $A$ of $I\times \Z$, 
$$
\mathbb{T}(A)
=
\min\{s\in\Z: (j,s) \in A\}.
$$
We are interested by the variable $T_{i,t}$ which is equal to $t-\mathbb{T}(A_{i,t})$ if the genealogy $G_{i,t}$ is non empty and equal to $0$ if $G_{i,t}$ is empty. 
By stationarity its distribution does not depend on $t$ and the behavior of this variable is of course linked to the one of the variables $T(V_j)=-\bT(V_j)$ for $V_j$ obeying the distribution $\lambda_j$, with the convention that $T(\emptyset)=0$. We are interested by conditions under which the variable $T_{i,t}$ has a Laplace transform, that is when 
$$\theta \mapsto \Psi_i(\theta)=\E(e^{\theta T_{i,t}})$$
is finite for some positive $\theta$. To do so, we are going to assume Assumption \ref{ass:exp_branch_proc} that we recall here:
There exists a strictly positive $\theta$ such that for all $i$,
$$\varphi_i(\theta)=\sum_{v\in \cV}  |v| e^{\theta T(v)}\lambda_i(v)$$
is finite and 
\begin{equation}\label{condlaplace}
 \varphi(\theta)=\sup_{i\in I} \varphi_i(\theta) <1,
\end{equation} 

\begin{thm}\label{laplace_gen}
Under Assumption \ref{ass:exp_branch_proc},
for all $i$ in $ I$, $\Psi_{i}(\theta)$ is finite and
$$\Psi(\theta)=\sup_{i \in I} \Psi_{i}(\theta) \leq \frac{\sup_{i\in I} \lambda_i(\emptyset)}{1-\varphi(\theta)}.$$
\end{thm}

Note that if $\varphi_i(\theta)$ is finite for some positive $\theta$, $\lim_{\theta\to 0}\varphi_i(\theta) =\bar{m_i}$. Therefore if Assumption \ref{ass:branch_proc_death} is fulfilled, $\lim_{\theta\to 0}\varphi_i(\theta) <1$ and it is possible to find $\theta>0$ such that $\varphi_i(\theta) <1$ as soon as $\lambda_i$ has a Laplace transform. In this sense, and roughly speaking, Assumption \ref{ass:exp_branch_proc} is a  more stringent condition of {\it probabilistic sparsity} than Assumption \ref{ass:branch_proc_death}.

\subsubsection{Application on the main examples}

\begin{paragraph}{Markov chains}
In this case, $\bar{m}=1-\mu$ and  the condition \eqref{ass:bar_m_small_1} is satisfied as soon as  $\mu<1$. Moreover condition  \eqref{condlaplace} reduces to $e^{\theta}(1-\mu)<1$ and it is always possible to find such a $\theta>0$ as soon as $\mu<1$.
\end{paragraph}

\begin{paragraph}{Chains of infinite order}
The space-time decomposition \eqref{def:kali_decom_inf_order} implies that
$$
\bar{m}=\sum_{\ell=1}^\infty\ell\lambda(\underline{\ell}).
$$
Thus, the condition \eqref{ass:bar_m_small_1} is satisfied as soon as $$\sum_{\ell=1}^\infty\ell\lambda(\underline{\ell})<1$$
and similarly the condition \eqref{condlaplace} is satisfied as soon as
$$ \sum_{\ell=1}^\infty\ell e^{\theta \ell}\lambda(\underline{\ell}) <1.$$
Hence both can be verified if $\lambda$ is  sufficiently exponentially decreasing. Typically one can have $\lambda(\underline{\ell})=e^{-\lambda}\lambda^{\ell}/\ell!$ with $0<\lambda<1$ (Poisson distribution on the range) or  $\lambda(\underline{\ell})=(1-p)^{\ell}p$ with $1/2<p\leq 1$ (Geometric distribution on the range).
\end{paragraph}

\begin{paragraph}{Discrete-time linear Hawkes processes}
According to the space-time decomposition \eqref{Kalikow_decomp_HP}, it follows that for each $i \in I$,
$$
m_i=\Sigma^+_i+\Sigma^-_i.
$$  
Therefore, the condition \eqref{ass:bar_m_small_1} reduces to 
$$
\sup_{i\in I}(\Sigma^+_i+\Sigma^-_i)= \sup_{i \in I} \sum_{j,s} |h_{j\to i}(-s)|<1.
$$
Moreover the condition \eqref{condlaplace} becomes
$$\sup_{i \in I} \sum_{j,s} e^{\theta s} |h_{j\to i}(-s)| <1.$$
So if for instance we can rewrite $h_{j\to i}(-s) =w_{j\to i} g(-s)$ for a fixed function $g$ of mean 1, the condition \eqref{ass:bar_m_small_1}  reduces to
$$\sup_{i \in I} \sum_{j\in I} |w_{j\to i}|<1,$$
and the additional  condition \eqref{condlaplace} is fulfilled for a small enough $\theta$ as soon as $g$ has finite exponential moment.
\end{paragraph}

\begin{paragraph}{GL neuron model}
In the nonlinear case, it has been proved in \cite{GL:13} (cf. inequalities (5.57) and (5.58)) that for each $i\in I$ the following estimates hold:
\begin{equation}
\lambda_i(\emptyset)\leq \gamma\sum_{j\in I}|W_{j\to i}|\sum_{s\geq 1}g_j(s),
\end{equation}
and for $\ell\geq 1$,
\begin{equation}
\lambda_i(v_i({\ell}))\leq \gamma\left(\sum_{j\notin v_i({\ell})}|W_{j\to i}|\sum_{s\geq 1}g_j(s)+\sum_{j\in v_i({\ell})}|W_{j\to i}|\sum_{s\geq \ell}g_j(s)\right).
\end{equation}
Therefore, a sufficient condition (cf. inequality (2.9) of \citep{GL:13}) for Assumption \ref{ass:branch_proc_death} to hold is
\begin{equation*}
\sup_{i\in I}\sum_{\ell\geq 1}\ell |v_i({\ell})|\left(\sum_{j\notin v_i({\ell}) }|W_{j\to i}|\sum_{s\geq 1}g_j(s)+\sum_{j\in v_i({\ell})}|W_{j\to i}|\sum_{s\geq \ell}g_j(s)\right)<\frac{1}{\gamma}.
\end{equation*}

In the linear case (i.e. when $\varphi_i(u)=\nu_i+u$), the condition above reduces to
\begin{equation}
\label{ass:bar_m_small_1_for_linear_GL}
\sup_{i\in I}\sum_{\ell\geq 1}\ell |v_i({\ell})|\left(\sum_{j\notin v_i({\ell}) }|W_{j\to i}|\sum_{s\geq 1}g_j(s)+\sum_{j\in v_i({\ell})}|W_{j\to i}|\sum_{s\geq \ell}g_j(s)\right)<1.
\end{equation}

Using the decomposition \eqref{ST_kalikow_decom_linear_GL}, one can verify that the condition \eqref{ass:bar_m_small_1} is, in the linear case, equivalent to
\begin{equation}
\label{ass:bar_m_small_1_for_gl_linear}
\sup_{i\in I}\sum_{\ell\geq 1}\left[\ell |W_{i\to i}|g_i(\ell)+\sum_{j\neq i, j\in I}(\ell+1)|W_{j\to i}|g_j(\ell)\right]<1.
\end{equation}
Note that  condition \eqref{ass:bar_m_small_1_for_linear_GL} is usually much stronger  than condition \eqref{ass:bar_m_small_1_for_gl_linear} and that a sparse space-time decomposition of the process allows us to derive existence of the linear process on a larger set of possible choices for $w_{j\to i}$ and $g_j$.
Once again condition \eqref{condlaplace} is fulfilled under a very similar expression
$$\sup_{i\in I}\sum_{\ell\geq 1} e^{\theta \ell}\left[\ell |W_{i\to i}|g_i(\ell)+\sum_{j\neq i, j\in I}(\ell+1)|W_{j\to i}|g_j(\ell)\right]<1,$$
this can be easily fulfilled if $g_j(\ell)=g(\ell)$ is exponentially decreasing with $\sum_{\ell=1}^\infty (\ell+1) g(\ell)=1$. Indeed \eqref{ass:bar_m_small_1_for_gl_linear} is implied as in the Hawkes case by 
$$\sup_{i\in I} \sum_{j\in I} |W_{j\to i}| <1$$
and it is easy to find by continuity a small $\theta>0$ such that \eqref{condlaplace} is fulfilled too.

\end{paragraph}

\subsection{Concentration}

\subsubsection{Block construction} 

Thanks to the control of the time length genealogy it is possible to cut  the observations $X_{F,-(m-1):T}$ into (overlapping) blocks that form with high probability two families of independent variables. This is a key tool to derive concentration inequalities. This construction is inspired by \cite{V:97}, who used as a central element, Berbee's lemma,  which is replaced here by Theorem \ref{laplace_gen}.
Note that similar coupling arguments have been used in continuous and more restrictive settings (see \cite{RBR:07,hrbr} for linear Hawkes processes, \cite{CSSW:17} for bounded Hawkes process and mixing arguments).

\begin{lemma}
\label{Lemma:1}
Let $m\in \Z_+$ and $F\subset I$ be a finite subset of the  neurons, observed on $-(m-1):T$. Let $B$, the grid size, be an   integer such that $$m\leq B\leq \lfloor T/2\rfloor$$ and define  $k=\lfloor \frac{T}{2B}\rfloor$. Let the $2k+1$ blocks be defined by, for $1\leq n\leq 2k$, $$I_n=\{(n-1)B+1-m,\ldots nB\} \mbox{ and } I_{2k+1}=\{2kB+1-m,\ldots T\}.$$
There exist on a common probability space  some stochastic chains ${\bf X}$, ${\bf X^1}$,...,${\bf X^{2k+1}}$  satisfying the following properties:
\begin{enumerate}
\item All the chains ${\bf X^n}=(X^n_{i,t})_{i\in I, t\in\Z}$ have the same distribution as ${\bf X}$ which satisfies Assumptions \ref{Mod_Assump:1}, \ref{Mod_Assump:2} and \ref{ass:exp_branch_proc} for a given $\theta$,  that is a sparse enough space-time decomposition with weights $(\lambda_i)_{i\in I}$ and transitions $(p_i^v)_{i\in I, v\in \cV}$.
\item  The odd chains ${\bf X^1},{\bf X^3},...,{\bf X^{2k+1}}$ are independent.
\item The even chains ${\bf X^2},...,{\bf X^{2k}}$ are independent.
\item There exists an event, $\Omega_{good}$, such that  on $\Omega_{good}$, $X_{F,I_n}=X^n_{F,I_n}$ for all $n=1,...,2k+1$ and
such that the probability of $\Omega^c_{good}$, under the notation of Theorem \ref{laplace_gen}, is at most
\begin{equation}
\label{prof_omega_good_complement}
|F|\left(2k+1\right)\frac{\Psi(\theta)}{(1-e^{-\theta})}e^{-\theta(B+1-m)}.
\end{equation}
\end{enumerate}
In particular, by choosing $B= m +\theta^{-1} ( 2 \log(T)+ \log(|F|))$, we obtain that there exists a positive $c'(\theta)$ such that the probability of $\Omega^c_{good}$ is at most
$c'(\theta)T^{-1}.$
\end{lemma}

\subsubsection{Applications}
As an application of Lemma \ref{Lemma:1}, we can derive the following Hoeffding type concentration inequality.
\begin{thm}
\label{Hoeff}
Let ${\bf X}=(X_{i,t})_{i\in I,t\in \Z}$ be a stationary sparse space-time process satisfying Assumptions \ref{Mod_Assump:1}, \ref{Mod_Assump:2} and \ref{ass:exp_branch_proc} for a given $\theta$. For $F\subset I$ finite, $m\in\Z_+$, let $f$ be a real-valued function  of $X_{F,t-m:t-1}$ bounded by $M$.
Let $T \in \Z_+$ such that 
$$m +\theta^{-1} ( 2 \log(T)+ \log(|F|))\leq \lfloor T/2\rfloor $$
and
\begin{equation}
\label{def:sum_of_centered_variable}
Z(f)=\frac{1}{T}\sum_{t=1}^T \left(f(X_{F,t-m:t-1})-\E\left[ f(X_{F,t-m:t-1})\right]\right).
\end{equation}
Then there exists nonnegative constant $c',c"$, which only depends on $\theta$ such that, for any $x>0,$ 
\begin{equation}
\label{Concentration_sup}
\P\left(Z(f)> \sqrt{c"(\theta) M^2\frac{m+\log T+\log|F|}{T}  x  }\right)\leq \frac{c'(\theta)}{T}+2e^{-x}.
\end{equation}
If there is a finite family $\mathcal{F}$ of such $f$, we also have that
$$\P\left(\exists f \in \mathcal{F}, Z(f)>\sqrt{c"(\theta) M^2\frac{m+\log T+\log|F|}{T}  x  }\right)\leq \frac{c'(\theta)}{T}+2|\mathcal{F}| e^{-x}.$$
\end{thm}

There is a matrix counterpart to the previous inequality, which is an application of now classical results on random matrices (see \cite{Tropp2012} and the references therein).

\begin{thm}
\label{Hoeff_Matrices}
Let ${\bf X}=(X_{i,t})_{i\in I,t\in \Z}$ be a stationary sparse space-time process satisfying Assumptions \ref{Mod_Assump:1}, \ref{Mod_Assump:2} and \ref{ass:exp_branch_proc} for a given $\theta$. For $F\subset I$ finite, $m\in\Z_+$, let $\cF$ be a finite family of bounded real-valued functions of $X_{F,t-m:t-1}$ and denote $M=\max\{\|fg\|_{\infty}:f,g\in\cF\}$.
Let $T \in \Z_+$ such that 
$$m +\theta^{-1} ( 2 \log(T)+ \log(|F|))\leq \lfloor T/2\rfloor $$
and define the random matrix
$Z=(Z(f,g))_{f,g\in\cF}$ where for each $f,g\in\cF$,
\begin{equation}
\label{def:matrix_entries_centered_variable}
Z(f,g)=\frac{1}{T}\sum_{t=1}^T \left(f(X_{F,t-m:t-1})g(X_{F,t-m:t-1})-\E\left[ f(X_{F,t-m:t-1})g(X_{F,t-m:t-1})\right]\right).
\end{equation}
Then there exists nonnegative constant $c',c"$, which only depends on $\theta$ such that, for any $x>0,$ 
\begin{equation}
\label{Concentration_sup}
\P\left(\|Z\|>\sqrt{c"(\theta)M^4 |\cF|^2 \frac{m+\log T+\log|F|}{T}x}\right)\leq \frac{c'(\theta)}{T}+4|\cF|e^{-x},
\end{equation}
where $\| Z\|$ corresponds to the spectral norm, that is the largest eigenvalue of the self-adjoint matrix $Z$.
\end{thm}

\section{Back to the Gram matrices\label{back}}

To control the Gram matrix we need also Assumption \ref{minmu} that we recall here:
There exists some positive $\mu$ such that for all $i\in I$, for all $x$,
$$\mu \leq p_i(x)\leq 1-\mu,$$

Note that in each of the examples (Markov chain, Hawkes, etc), this assumption is easily fulfilled. For instance in the Hawkes case, this adds the condition $\mu\leq \nu_i-\Sigma_i^- \leq \nu_i+\Sigma_i^+\leq (1-\mu)$.

This assumption is useful to bound expectation by changing the underlying measure.

\begin{lemma}
\label{likelihood_lower_bound}
Under Assumptions \ref{Mod_Assump:1} and  \ref{minmu},  for all non negative function $f$ cylindrical on a fixed finite space-time neighborhood $v$, 
$$(2(1-\mu))^{|v|} \E_{\mathcal{B}(1/2)}^{\otimes \mathcal{V}}\left[f(X_v)\right] \geq \E\left[f(X_v)\right] \geq (2\mu)^{|v|} \E_{\mathcal{B}(1/2)}^{\otimes \mathcal{V}}\left[f(X_v)\right],$$
where $\E_{\mathcal{B}(1/2)}^{\otimes \mathcal{V}}$ means that the expectation is taken with respect to the measure where all $X_{i,t}$'s are i.i.d Bernoulli with parameter $1/2$.
\end{lemma}

\subsection{Inv($\kappa$) property for general dictionaries}

In this section  we prove that the Inv($\kappa$) property holds on an event with high probability for the examples of dictionaries considered in Subsection \ref{Ex_dict}.
As a by product, we are able to derive oracle inequalities with high probability for these dictionaries. We start with the following result.

\begin{thm}
\label{inv(k)_High_Prob}
For a finite $F\subset I$ and integer $T>m\geq 1$, let $X_{F,-(m-1):T}$ be a sample produced by the stationary sparse space-time process ${\bf X}=(X_{i,t})_{i\in I,t\in \Z}$ satisfying Assumptions \ref{Mod_Assump:1}, \ref{Mod_Assump:2} and  \ref{ass:exp_branch_proc} . Let $\Phi$ denote a finite dictionary of bounded functions cylindrical in $F\times \underline{m}$ and $G$ be the corresponding Gram matrix defined in \eqref{def:b_and_G}. 
If the matrix $\E(G)$ satisfies property $Inv(\kappa')$ 
for some positive constant $\kappa',$ then for any $\delta>0$ and $T$ sufficiently large, the Gram matrix $G$ satisfies the property $Inv(\kappa)$ on an event of probability larger than $1-\frac{c'}{T}-\delta$ with 
$$
\kappa=\kappa'
-c_1|\Phi|\|\Phi\|^2_{\infty}\sqrt{\frac{(m+\log(T)+\log|F|)(\log|\Phi|+\log\delta^{-1})}{T}},
$$
where $c'$ and $c_1$  are positive constants  which only depends on the underlying distribution of ${\bf X}$.
\end{thm}

To apply Theorem \ref{inv(k)_High_Prob} to the dictionaries considered in Section \ref{Ex_dict} we must find the corresponding $\kappa'$. This is done below.

\begin{paragraph}{Short memory effect}
To apply Theorem \ref{inv(k)_High_Prob} we need first to find $\kappa'$ for this class of models. This is done as follows.  Let  $Q=\mathcal{B}(1/2)^{\otimes \mathcal{V}}$ be the probability measure under which all $X_{i,t}$'s are i.i.d. Bernoulli with parameter $1/2$ and denote $p_j=Q( \varphi_{j}(X_{-\infty:-1})=1)$ for $j\in F.$
Clearly, $p_j=1-(1/2)^m$ for all $j\in F$ and we write $p$ to denote this common value. With this notation, one can check that,
$$\E_{\mathcal{B}(1/2)}^{\otimes \mathcal{V}}(G)= \begin{pmatrix} p & p^2& p^2& ...& p^2 \\ p^2& p& p^2& ...& p^2\\&&&...&\\p^2 &p^2& p^2& ... &p \end{pmatrix}.$$ 
Such a matrix has only two eigenvalues, namely,  $p+(|F|-1) p^2$ of multiplicity 1 and $p-p^2=(1/2)^m (1-(1/2)^m)$ with multiplicity $|F|-1$. 
Indeed, $\xi$ is an eigenvalue $\E_{\mathcal{B}(1/2)}^{\otimes \mathcal{V}}(G)$ if and only if there exists a non-null vector 
$u\in \R^{F}$ such that
$$
(p-p^2)u+p^2\sum_i u_i {\bf 1}=\xi u.
$$
On the one hand, by choosing the vector $u\neq 0$ such that $\sum_i u_i=0$ gives that $\eta=p-p^2$ is an eigenvalue with multiplicity $|F|-1$. On the other hand, the choice $\sum_i u_i=1$ forces that $(p-p^2)u_i+p^2=\xi u_i$ for all $i\in F$, ensuring that $\xi=p+p^2(|F|-1)$ is the second eigenvalue. Its multiplicity is necessarily $1.$

Note that if $m$ is large, the smallest eigenvalue of $\E_{\mathcal{B}(1/2)}^{\otimes \mathcal{V}}(G)$ is really small. This can be interpreted in the following way : when $m$ is large, one will find a "1" on every observed neuron in the past, therefore all the $\varphi_j$'s will be equal with high probability and one cannot infer a dependence graph with this dictionary anymore.

Thus,  Lemma \ref{likelihood_lower_bound} implies that eigenvalue of $\E(G)$ can be lower bounded by
\begin{equation}
\label{kappa_prime_Short_effect}
\kappa'=(2\mu)^{m|F|} (1/2)^m (1-(1/2)^m).
\end{equation}
Choosing for a fixed integer $\eta$
\begin{equation}\label{choixshort}
m= \eta \mbox{ and } |F|\leq \log\log T,
\end{equation}
gives $\kappa'$ of the order $(\log(T))^{-c_3}$ for some constant $c_3>0$ depending on $\mu$ and $\eta$.

\end{paragraph}

\begin{paragraph}{Cumulative effect}
Let $\alpha$ denote the common value of $\E_{\mathcal{B}(1/2)}^{\otimes \mathcal{V}}(\varphi^2_{j,\ell}(X_{-\infty:-1}))$ with $j\in F$ and $1\leq \ell\leq L$, and $\beta$ be the corresponding value of $\E_{\mathcal{B}(1/2)}^{\otimes \mathcal{V}}(\varphi_{j,\ell}(X_{-\infty:-1})(\varphi_{k,n}(X_{-\infty:-1}))$ with $j,k\in F$ and $k\neq j$ and $1\leq n,\ell \leq L$.  With this notation, one can verify that
$$
\alpha=\frac{\eta}{2}+\frac{\eta(\eta-1)}{4}=\frac{\eta}{4}+\frac{\eta^2}{4}, \  \beta=\frac{\eta^2}{4} \ \mbox{and} \  \E_{\mathcal{B}(1/2)}^{\otimes \mathcal{V}}(G)=\begin{pmatrix} \alpha & \beta & \beta& ...& \beta \\ \beta& \alpha & \beta& ...& \beta\\&&&...&\\\beta &\beta&...& \beta &\alpha \end{pmatrix}.
$$
Hence, the smallest eigenvalue of $\E_{\mathcal{B}(1/2)}^{\otimes \mathcal{V}}(G)$ is $ \alpha-\beta= \frac{\eta}{4}$ which grows with $\eta=\frac{m}{K}$. This seems also reasonable since once looking for cumulative effects, the larger the bin size $\eta$, the more points you see in it and the more diverse the situations are (hence the dictionary has many different functions) whereas if $\eta$ is small there is a large probability to see all $\varphi_{j,\ell}$'s null.

Thus,  Lemma \ref{likelihood_lower_bound} implies that eigenvalue of $\E(G)$ can be lower bounded by
$$\kappa'=\frac{\eta}{4}(2\mu)^{\eta
K|F|}.$$

Choosing  for some fixed integer $\eta$ 
\begin{equation}\label{choixcum}
m= \eta  K \mbox{ with } K\leq \sqrt{\log\log T} \mbox{ and }|F|\leq \log\log T,
\end{equation}
gives $\kappa'$ of the order $(\log(T))^{-c_3}$ for some other constant $c_3>0$ depending on $\mu$ and $\eta$.

\end{paragraph}

\
\begin{paragraph}{Cumulative effect with spontaneous apparition}
With the same notation of the previous example,  one can show that
$$\E_{\mathcal{B}(1/2)}^{\otimes \mathcal{V}}(G)=\begin{pmatrix} 1 & \eta/2 &\eta/2 &  ... & \eta/2\\
\eta/2 & \alpha & \beta &  ...& \beta \\&&...&& \\ \eta/2 & \beta &...& \beta &\alpha \end{pmatrix}.$$
Reasoning by block with the vector $(\mu,a)$  with $\mu\in \R$ and $a\in \R^{K|F|}$, we end up with
$$(\mu,a)^\intercal \E_{\mathcal{B}(1/2)}^{\otimes \mathcal{V}}(G) (\mu,a) = \left(\mu+\frac{\eta}{2} \sum_{j\in F, k=1,...,K} a_{j,k}\right)^2 + \frac{\eta}{4} \| a\|_2^2.$$ 
But for all $0<\theta<1$, 
\begin{eqnarray*}
 \left(\mu+\frac{\eta}{2} \sum a_{j,k}\right)^2 + \frac{\eta}{4} \| a\|_2^2 & \geq & (1-\theta)\mu^2 + \left(1-\frac{1}{\theta}\right) \frac{\eta^2}{4} \left(\sum_{j\in F, k=1,...,K} a_{j,k}\right)^2+ \frac{\eta}{4} \| a\|_2^2 \\
 &\geq &  (1-\theta)\mu^2 - \frac{1-\theta}{\theta} \frac{K|F|\eta^2}{4} \| a\|_2^2 + \frac{\eta}{4} \| a\|_2^2 
\end{eqnarray*}
By choosing $\theta= \frac{2 \eta K |F|}{1+2 \eta K |F|}$ we conclude, thanks to Lemma \ref{likelihood_lower_bound}, that the smallest eigenvalue of $\E(G)$ can be lower bounded by
$$ \kappa'=(2\mu)^{\eta K |F|} \min\left(\frac{1}{1+2 \eta K |F|}, \frac{\eta}{8}\right).$$
Once again choosing  for some fixed integer $\eta$ 
\begin{equation}\label{choixcumspont}
m= \eta  K \mbox{ with } K\leq \sqrt{\log\log T} \mbox{ and }|F|\leq \log\log T,
\end{equation}
gives $\kappa'$ roughly larger than $(\log(T))^{-c_3}$ for some other constant $c_3>0$ depending on $\mu$ and $\eta$.

\end{paragraph}

Next, as a by product of Theorem \ref{inv(k)_High_Prob} and Theorem \ref{thm:Oracle_inequality}, one can derive oracle inequalities for dictionaries above.

\begin{corollary}
\label{Oracle_HP_Short_effect}
Let $\Phi$ be one of the dictionaries presented in Section \ref{Ex_dict}, with the choices \eqref{choixshort}, \eqref{choixcum} or \eqref{choixcumspont}. Assume one observes  $X_{F,-(m-1):T}$, where the underlying process ${\bf X}$ satisfies Assumptions \ref{Mod_Assump:1}, \ref{Mod_Assump:2}, \ref{ass:exp_branch_proc} and \ref{minmu}.

With the notation of Theorem \ref{thm:Oracle_inequality}, for  $T$ large enough,  on an event with probability $1-c_1/T$, the following oracle inequality holds

$$
\| \hat{f}-p_i(\cdot) \|^2_T\leq   \inf_{a\in\R^{\Phi}}\left\{\|f_a-p_i(\cdot) \|^2_T +c_2 |S(a)|\frac{(\log(T))^{c_3}}{T}\right\},
$$
where the constant $c_1>0$ depends only on the underlying distribution of {\bf X}, $c_2>0$ depends on $\eta$ and $\gamma$  and constant $c_3>0$ depends on both the underlying distribution of {\bf X} and $\eta$.
\end{corollary}

Note that the main improvement with respect to \cite{hrbr}, is that in all the examples , the constant $\kappa$ is roughly of order $(\log(T))^{-c_3}$, that is asymptotically decreasing in roughly speaking the number of neurons used in the dictionary and not the total number of neurons in the network. This number of neurons that are used, which is bounded by the number of observed neurons, can very slowly grow with $T$.

\subsection{Hawkes  dictionary without  spontaneous part}
In this case the $\varphi(X_{F,-m:-1})$'s are just the $X_{j,s}$ for $j\in F$ and $s\in \underline{m}$ and one can prove the following result. 
\begin{thm}
\label{REHawkes}
For a finite $F\subset I$ and integer $T>m\geq 1$, let $X_{F,-(m-1):T}$ be a sample produced by the stationary sparse space-time process ${\bf X}=(X_{i,t})_{i\in I,t\in \Z}$ satisfying Assumptions \ref{Mod_Assump:1}, \ref{Mod_Assump:2}, \ref{ass:exp_branch_proc} and \ref{minmu}.
For the Hawkes dictionary without spontaneous part, i.e. $\varphi=\varphi_{j,s}$ with $\varphi_{j,s}(X_{F,-m:-1}) = X_{j,s}$ for $j\in F$ and $s\in \underline{m}$, the corresponding Gram matrix $G$ defined by \eqref{def:b_and_G} satisfies  for all $c>0$, $s \leq  m|F|$ and $T$ large enough, the property ${\bf RE}(\kappa,c,s)$ on an event  of probability larger than $1-\frac{c'}{T}-\delta$ with
$$ \kappa =  \mu-\mu^2-((1-2\mu)+R_T)(1+c) s,$$
where 
$$R_T= \frac{c_1}{T^{1/2}} (m+\log T+ \log|F|)^{1/2}(\log m +\log |F| +\log \delta^{-1})^{1/2},$$
for some positive constant $c'$ and $c_1$  which only depends on the underlying distribution of ${\bf X}$. 
\end{thm}

The major point to note is that asymptotically, for slowly growing $m$ and $|F|$ as functions of $T$, the constant $\kappa$ does not depend at all on the number of observed neurons and therefore the rate of convergence in Theorem \ref{thm:Oracle_inequality} is not worsened by a huge number of observed neurons, $|F|$. This is a drastic improvement  with respect to the previous result of  \cite{hrbr} which depends on the total number of neurons in the network. For each fixed $c$ and $s$, we only need here $\mu$  to be close enough to $1/2$ to have  $\kappa>0$. 

It also means that the size of the dictionary might be growing with $T$, much more rapidly than before: typically $m$ the delay might grow like $\log(T)$ and the number of observed neurons might grow  like $T$ or even more rapidly as long as $\log|F|= o(T^{1/2})$. Therefore if one can reasonably well  approximate $p_i$ by a sparse combination in space and time for which the precise location is unknown, one might by a growing set of observations find  the correct set in space and time.

\section{Conclusion\label{conclusion}}
~\\

It is therefore possible to control the Gram matrix for various dictionaries and this even if the finite number of observed neurons is much smaller than the potentially infinite set of existing neurons. The main assumption on the underlying stochastic structure is the {\it probabilistic sparsity} (Assumption \ref{ass:exp_branch_proc}) which allows us to derive concentration inequalities via coupling.

As an open question, it remains to understand the complete link between a well chosen deterministic sparse approximation of $p_i$  and the {\it probabilistic sparsity} of the $\lambda_i$'s typically when both the approximation model and the true underlying model coincide, for instance for Hawkes processes. Another way to phrase this is ``can we prove the variable selection property, that is typical of Lasso methods? " i.e. `` can we find the set of neurons influencing $i$ ?".  If the answer seems likely to be yes if they are all observed, it seems intuitive to think in general that a good set of sites $(j,s)$ for the sparse approximation of $p_i$ is a level set of the $\lambda_i$ but the fact that the $\lambda_i$'s are not unique makes this reasoning not straightforward.

Another open question is the minimax rate in this setting. This would involve speaking about regularity of the space-time decomposition, which is not done yet, since we do not even have for the moment uniqueness of the decomposition.

\section{Proofs}
\label{Sec:proof}
\subsection{Proof of Theorem \ref{thm:Oracle_inequality}}

To prove Theorem \ref{thm:Oracle_inequality} we use arguments from \cite{gaiffas2012}. We will need the following Lemmas.

\begin{lemma}
\label{Lemma:Sharp_oracle}
Let $\hat{f}=f_{\hat{a}}$ where $\hat{a}$ is defined by \eqref{lasso}. For any vector $a\in\R^{\Phi}$, the following inequality holds
\begin{equation}
2\langle \hat{f}-f_a, \hat{f}-p_i \rangle_T + \gamma d|\hat{a}_{S^c(a)}|_1\leq \gamma d |\hat{a}_{S(a)}-a_{S(a)}|_1+2 (b-\bar{b})^T(\hat{a}-a),
\end{equation}
where $S(a)=\{\varphi:a_{\varphi}\neq 0\}$ and the vectors $b,\bar{b}\in \R^{\Phi}$ are defined in \eqref{def:b_and_G} and \eqref{def:barb} respectively.
\end{lemma}
\begin{proof}
Throughout the proof we write $\partial g(p)$ to denote the subddiferential mapping of a convex function g at the point $p$.
One can show that $p$ is a global minimum of the convex function $g$ if and only if $0\in \partial g(p)$. Now since $\hat{a}$ is such that 
$$
\hat{a}\in \  \arg\min_{a\in\R^{\Phi}}\{a^TGa-2a^Tb+\gamma d|a|_1\},
$$
it follows that
$$
0 \in \partial (\hat{a}^TG\hat{a}-2\hat{a}^Tb+\gamma d|\hat{a}|_1)=2G\hat{a}-2b+\gamma d \partial |\hat{a}|_1.
$$
Thus,  it follows that for some $\hat{w}\in \partial|\hat{a}|_1$, the following equation holds
$$
2G\hat{a}-2b+\gamma d\hat{w}=0,
$$
which implies then
\begin{equation*}
(2G\hat{a}-2b+\gamma d\hat{w})^T(\hat{a}-a)=0, \ \mbox{for any} \ a\in\R^{\Phi}.
\end{equation*}
From the above equation we can deduce that for any vector $w\in\partial |a|_1$ and $a\in\R^{\Phi}$,
\begin{equation}
\label{Lemma_1:eq_1}
(2G\hat{a}-2\bar{b})^T(\hat{a}-a)+\gamma d (\hat{w}-w)^T(\hat{a}-a)=-\gamma d w^T(\hat{a}-a)+2(b-\bar{b})^T(\hat{a}-a).
\end{equation}
One can easily show by the definition of subdifferentials that 
$$
(\hat{w}-w)^T(\hat{a}-a)\geq 0,
$$
for all $\hat{w}\in |\hat{a}|_1$ and $w\in |a|_1.$ Thus, using this fact in equation \eqref{Lemma_1:eq_1} together with the fact that
$(2G\hat{a}-2\bar{b})^T(\hat{a}-a)=2\langle \hat{f}-f_a, \hat{f}-p_i \rangle_T$, we derive the following inequality
\begin{equation}
\label{Lemma_1:eq_2}
2\langle \hat{f}-f_a, \hat{f}-p_i \rangle_T\leq  -\gamma d w^T(\hat{a}-a)+2(b-\bar{b})^T(\hat{a}-a).
\end{equation}
It is well know that
$$
\partial |a|_1=\{v: |v|_{\infty}\leq 1 \ \mbox{and} \ v^Ta=|a|_1\}.
$$
In other words, $v \in \partial |a|_1$ if and only if $v_{\varphi}=\mbox{sign}(a_{\varphi})$ for ${\varphi}\in S(a)$ and $v_{\varphi}\in [-1,1]$ for all ${\varphi}\in S^c(a).$ Now, take $w=(w_{\varphi})_{\varphi\in\Phi}\in \partial |a|_1 $ of the following form
$$
w_{\varphi}=
\begin{cases}
\mbox{sign}(a_{\varphi}), \ \mbox{if} \ {\varphi}\in S(a)\\
\mbox{sign}(\hat{a}_{\varphi}), \ \mbox{if} \ {\varphi}\in S^c(a)
\end{cases},
$$
and observe that 
$w^T(\hat{a}-a)=\sum_{\varphi\in S(a)}\mbox{sign}(a_{\varphi})(\hat{a}_{\varphi}-a_{\varphi})+ |\hat{a}_{S^c(a)}|_1$. Thus, by plugging this identify into inequality \eqref{Lemma_1:eq_2}, we obtain that
$$
2\langle \hat{f}-f_a, \hat{f}-p_i \rangle_T +\gamma d|\hat{a}_{S^c(a)}|_1 \leq -\gamma d\sum_{\varphi\in S(a)}\mbox{sign}(a_{\varphi})(\hat{a}_{\varphi}-a_{\varphi})+2(b-\bar{b})^T(\hat{a}-a),
$$
and the result follows, because $|-\sum_{\varphi\in S(a)}\mbox{sign}(a_{\varphi})(\hat{a}_{\varphi}-a_{\varphi})|\leq |\hat{a}_{S(a)}-a_{S(a)}|_1.$
\end{proof}

\begin{lemma}
\label{Lemma:RE_property_sharph_oracle}
Let $\hat{f}=f_{\hat{a}}$ where $\hat{a}$ defined by \eqref{lasso} with $\gamma\geq 2$ and $a\in\R^{\Phi}$.
On an event on which 
\begin{enumerate}
\item[(i)] $\langle \hat{f}-f_a, \hat{f}-p_i \rangle_T\geq 0$,
\item[(ii)] $|b_{\varphi}-\bar{b}_{\varphi}|\leq d$ for all $\varphi\in\Phi,$
\end{enumerate}
the following inequality is satisfied,
\begin{equation}
|\hat{a}_{S^c(a)}|_1\leq \frac{\gamma+2}{\gamma-2}|\hat{a}_{S(a)}-a_{S(a)}|_1,
\end{equation}
where $S(a)=\{\varphi:a_{\varphi}\neq 0\}$.
\end{lemma}

\begin{proof}
Suppose that $\langle \hat{f}-f_a, \hat{f}-p_i \rangle_T\geq 0$. In this case,   Lemma \ref{Lemma:Sharp_oracle} implies that
$$
\gamma d |\hat{a}_{S^c(a)}|_1\leq \gamma d |\hat{a}_{S(a)}-a_{S(a)}|_1+2\sum_{\varphi\in S(a)}(b_{\varphi}-\bar{b}_{\varphi})(\hat{a}_{\varphi}-a_{\varphi})+2\sum_{\varphi\in S^c(a)}(b_{\varphi}-\bar{b}_{\varphi})\hat{a}_{\varphi}.
$$
On an event on which $|b_{\varphi}-\bar{b}_{\varphi}|\leq d$ for all $\varphi\in\Phi$, we then have that
$$
\gamma d |\hat{a}_{S^c(a)}|_1\leq (\gamma +2)d |\hat{a}_{S(a)}-a_{S(a)}|_1+2d |\hat{a}_{S^c(a)}|_1,
$$
and the result follows.
\end{proof}

To prove the first part of Theorem \ref{thm:Oracle_inequality} we proceed as follows. First of all, 
on the event on which $\langle \hat{f}-f_a, \hat{f}-p_i \rangle_T<0$, there is nothing to be proved, since in this case
$$
\|\hat{f}-p_i\|^2_T+\|\hat{f}-f_{a}\|^2_T-\|f_a-p_i\|^2_T= \langle \hat{f}-f_a, \hat{f}-p_i \rangle_T<0.
$$
Hence, in what follows, take $a=(a_{\varphi})_{\varphi\in\Phi}$ such that $|S(a)|\leq s$ and $\langle \hat{f}-f_a, \hat{f}-p_i \rangle_T\geq 0$. In this case,  thanks to Lemma \ref{Lemma:RE_property_sharph_oracle}, we can use 
Property ${\bf RE}(\kappa,c(\gamma),s)$ to the vector $\hat{a}-a:$
$$
\|\hat{a}_{S(a)}-a_{S(a)}\|^2\leq \kappa^{-1}(\hat{a}-a)^TG(\hat{a}-a).
$$
Now, as in the proof of Lemma \ref{Lemma:RE_property_sharph_oracle}, we know that on an event on which $|b_{\varphi}-\bar{b}_{\varphi}|\leq d$ for all $\varphi\in\Phi$, the following bound holds:
$$2|(b-\bar{b})^T(\hat{a}-a)|\leq 2d|(\hat{a}_{S(a)}-a_{S(a)})|_1+2d|\hat{a}_{S^c(a)}|_1
$$
By using this inequality together with Lemma \ref{Lemma:Sharp_oracle}, we conclude that
\begin{equation}
\label{Shape_oracle:ineq_1}
2\langle \hat{f}-f_a, \hat{f}-p_i \rangle_T + (\gamma-2) d|\hat{a}_{S^c(a)}|_1\leq (\gamma+2) d |\hat{a}_{S(a)}-a_{S(a)}|_1.
\end{equation}
Finally, by Cauchy-Schwartz inequality, we know that
$$
|\hat{a}_{S(a)}-a_{S(a)}|_1\leq \sqrt{S(a)}\|\hat{a}_{S(a)}-a_{S(a)}\|\leq \sqrt{S(a)\kappa^{-1}(\hat{a}-a)^TG(\hat{a}-a)}.
$$
Plugging this last inequality into \eqref{Shape_oracle:ineq_1}, we deduce that
$$
2\langle \hat{f}-f_a, \hat{f}-p_i \rangle_T + (\gamma-2) d|\hat{a}_{S^c(a)}|_1\leq (\gamma+2) d \sqrt{S(a)\kappa^{-1}(\hat{a}-a)^TG(\hat{a}-a)} .
$$
To conclude the proof of the first part, note that 
$$
\begin{cases}
2\langle \hat{f}-f_a, \hat{f}-p_i \rangle_T=\|\hat{f}-p_i\|^2_T+\|\hat{f}-f_{a}\|^2_T-\|f_a-p_i\|^2_T\\
(\hat{a}-a)^TG(\hat{a}-a)=\|\hat{f}-f_{a}\|^2_T,
\end{cases}
$$
and use the inequality $qy-y^2\leq q^2/4$, which is valid for  any $q,y>0.$

For the second part of the result, to control the fluctuations of $b_\varphi-\bar{b}_\varphi$, let us note that
$b_\varphi-\bar{b}_\varphi = M_T,$ where $(M_t)_{1\leq t\leq T}$ is the martingale defined by
$$ M_t = \sum_{i=1}^t \frac{\varphi(X_{-\infty:t-1})}{T} \left[X_{i,t}-p_i(X_{-\infty:t-1})\right].$$
We can apply the classical bound of Hoeffding's inequality on each increment of the martingale $\Delta M_t$. Note that if $\varphi(X_{-\infty:t-1})$ is positive, 
$$-\frac{\varphi(X_{-\infty:t-1})}{T} p_i(X_{-\infty:t-1}) \leq \Delta M_t \leq \frac{\varphi(X_{-\infty:t-1})}{T} [1- p_i(X_{-\infty:t-1})],$$
and if $\varphi(X_{-\infty:t-1})$ is negative,
$$\frac{\varphi(X_{-\infty:t-1})}{T} [1-p_i(X_{-\infty:t-1})] \leq \Delta M_t \leq - \frac{\varphi(X_{-\infty:t-1})}{T} p_i(X_{-\infty:t-1}).$$
This leads for every $\theta>0$ to
$$\E( e^{\theta \Delta M_t} | X_{-\infty:t-1}) \leq \exp\left(\frac{\theta^2 \varphi(X_{-\infty:t-1})^2}{8T^2}\right)\leq\exp\left(\frac{\theta^2 \|\Phi\|^2}{8T^2}\right) .$$
Therefore
$$\E( e^{\theta M_T}) \leq \exp\left(\frac{\theta^2 \|\Phi\|^2}{8T}\right).$$
Hence
$$ \P( M_T \geq x) \leq \exp\left(\frac{\theta^2 \|\Phi\|_\infty^2}{8T} -\theta x\right).$$
By optimizing this in $\theta$ and applying the same inequality to $-\varphi$, we get for all positive $u$
$$\P\left( M_T \geq \sqrt{\frac{u \| \Phi\|^2_\infty}{2T}}\right) \leq e^{-u} \mbox { and }\P\left(  |b_\varphi-\bar{b}_\varphi| \geq \sqrt{\frac{u \| \Phi\|^2_\infty}{2T}}\right) \leq 2 e^{-u}$$
Therefore taking $u =\log|\Phi| +\log(2\delta^{-1})$ and then applying the union  bound we obtain the result.

 \subsection{Proof of Proposition \ref{N_i_t_finite}}

Since $\{N_{(i,t)}>\ell\}=\{|A^{\ell}_{i,t}|\geq 1\}$, the Markov inequality implies that
\begin{eqnarray*}
\P\left(N_{i,t}>\ell\right)\leq \E\left[|A^{\ell}_{i,t}|\right].
\end{eqnarray*}
So let us prove by induction that $\E\left[|A^{\ell}_{i,t}|\right]\leq (\bar{m})^{\ell}$ for all $\ell\geq 1$. For $\ell=1$,  we have $\E\left[|A^{1}_{i,t}|\right]=\E\left[|V_{i,t}|\right]=\bar{m}_i\leq \bar{m}$. Next for $\ell>1$,
\begin{eqnarray*}
\E\left[|A^{\ell}_{i,t}|~|~A^{\ell-1}_{i,t}\right] &\leq & \sum_{(j,s)\in A^{\ell-1}_{i,t}} \E\left[|V_{j,s}^{\to s}|\right]\\
&\leq & \sum_{(j,s)\in C_{i,t}(\ell-1)} \bar{m}_j ~~
\leq ~  |A^{\ell-1}_{i,t}| ~\bar{m}.
\end{eqnarray*}
To conclude the proof take the overall expectation and use the induction assumption  given by  $\E\left[|A^{\ell-1}_{i,t}|\right]\leq (\bar{m})^{\ell-1}$.

\subsection{Proof of Theorem \ref{laplace_gen}}

For any fixed $n\geq 1$, for all site $(i,t)$ let
$$G_{i,t}^{n}=\cup_{m=1}^{n} A_{i,t}^{m}$$
We adopt the convention that if $G_{i,t}^{n}=\emptyset$, $\mathbb{T}( G_{i,t}^{n})=t$ and we consider the variable $T^n_{i,t}=t-\mathbb{T}( G_{i,t}^{n})$ as well as its Laplace transform $\Psi_i^n(\theta)=\E(e^{\theta T^n_{i,t}}).$

Let us prove by induction that $\Psi_i^n(\theta)$ is finite and that 
\begin{equation}\label{induc_psi_L}
\Psi^n(\theta)=\sup_{i} \Psi_{i}^n(\theta) \leq \bar{\lambda}(1+\varphi(\theta)+...+\varphi(\theta)^{n-2}){\bf 1}_{n>1}+\varphi(\theta)^{n-1} g(\theta),
\end{equation}
where $\bar{\lambda}= \sup_{i \in I} \lambda_i(\emptyset)$ and
$$ g(\theta) =\sup_{i\in I} \sum_{v \in \cV} e^{\theta T(v)} \lambda_i(v).$$
Note that $g(\theta)$ is finite as soon as $\varphi(\theta)$ is and that $0\leq \bar{\lambda}\leq 1.$

For $n=1$, since for all $i$, $\mathbb{T}( G^{1}_{i,t})=\mathbb{T}(A_{i,t}^{1})=\mathbb{T}(K_{i,t})=t-T(V_{i,t})$
\begin{eqnarray*}
\Psi_{i}^1(\theta)&=&\E\left(\exp\left[\theta T(V_{i,t})\right]\right)\\
&=&\sum_{v\in \cV}e^{\theta T(v)}\lambda_{i}(v)\\
&\leq & g(\theta). 
\end{eqnarray*}

Next by induction, let us assume \eqref{induc_psi_L} at level $n $ for all $i$ and let us prove it at level $n+1$.
Note that because the $G_{i,t}^{n}$ are computed recursively, we have that when $K_{i,t}$ is not empty,
$$\mathbb{T}(G_{i,t}^{n+1})=\min_{(k,r)\in K_{i,t}} \mathbb{T}(G_{k,r}^{n}).$$

Therefore if $K_{i,t}=\emptyset$, $T^{n+1}_{i,t}=0$ and
$$
\E\left( \exp\left[ \theta T^{n+1}_{i,t}\right] ~~ |~~ K_{i,t}\right) = 1.$$
This happens with probability $\lambda_j(\emptyset).$
If  $K_{i,t}\neq\emptyset,$
\begin{eqnarray*}
\E\left( \exp\left[ \theta \left(t - \mathbb{T}(G_{i,t}^{n+1})\right)\right] ~~ |~~ K_{j,t}\right) &=&\E\left( \exp\left[ \theta  \max_{(k,r)\in K_{i,t}} \left(t - \mathbb{T}(G_{k,r}^{n})\right)\right] ~~ |~~ K_{i,t}\right) \\
&\leq & \sum_{(k,r)\in K_{i,t}} e^{\theta(t-r)}\E\left( \exp\left[ \theta \left(r- \mathbb{T}(G_{k,r}^{n})\right)\right] ~~ |~~ K_{i,t}\right).
\end{eqnarray*}

Since (see the algorithm) $K_{i,t}$ only depends on $U^1_{j,t}$ and $G^{n}_{k,r}$ only depends on the $U^1_{k',r'}$ for $k'\in I,r'\leq r$ and $r<t$, it follows that $\bT(G^{n}_{k,r})$ is independent of $K_{i,t}.$
Hence if $K_{i,t}\neq\emptyset$
\begin{eqnarray*}
\E\left( \exp\left[ \theta T^{n+1}_{i,t} \right] ~~ |~~ K_{j,t}\right) & \leq & \sum_{(k,r)\in K_{i,t}} e^{\theta(t-r)} \Psi_{k}^n(\theta) \\
& \leq & \left[ \sum_{(k,r)\in K_{i,t}} e^{\theta(t-r)} \right] \Psi^n(\theta)\\
&\leq & \left[|K_{i,t}| e^{\theta (t- \bT(K_{j,t}))} \right] \Psi^n(\theta) \\
& \leq & |V_{i,t}| e^{\theta T(V_{i,t})} \Psi^n(\theta).
\end{eqnarray*}
We obtain by taking the overall expectation that 
$$ \Psi_{i}^{n+1}(\theta) \leq \bar{\lambda} + \varphi(\theta) \Psi^n(\theta),$$
so that $\sup_{i\in I} \Psi_{i}^{n+1}(\theta)$ is finite and  \eqref{induc_psi_L} holds at level $n+1$ by induction.

To conclude, it is sufficient to remark that by the monotone convergence theorem, $\Psi_{i}^n(\theta)\to_{n\to \infty} \Psi_{i}(\theta)$ which are therefore upper bounded by $\bar{\lambda}/(1-\varphi(\theta)).$ This concludes the proof.

\subsection{Proof of Lemma \ref{Lemma:1}}

We use the perfect simulation algorithm  to construct these chains. 
Let ${\bf U^{0}}=(U^{0,1}_{i,t}, U^{0,2}_{i,t})_{i\in I,t\in \Z},\ldots, {\bf U^{2k+1}}=(U^{2k+1,1}_{i,t}, U^{2k+1,2}_{i,t})_{i\in I,t\in \Z}$ be independent fields of independent random variables with uniform distribution  on $[0,1]$. We assume that these sequences are defined in the same probability space and set $(\tilde{\Omega}, \tilde{\cF}, \tilde{\P})$ to be this common probability space.

The perfect simulation algorithm performed  with the  same field ${\bf U^{0}}$ on each site $(i,t)$ yields  the construction of ${\bf X}=(X_{i,t})_{i\in I , t\in \Z}$.

For any $n$, the chain ${\bf X^n}$ is also built similarly via the perfect simulation algorithm but with the field ${\bf U^n}$ except on a small portion of time where we use ${\bf U^0}$. More precisely, we use the following  variables
$$
\left((U^{n,1}_{i,t},U^{n,2}_{i,t})_{i\in I,t\leq (n-2)B}, (U^{0,1}_{i,t},U^{0,2}_{i,t})_{i\in I,(n-2)B<t \leq nB}, (U^{n,1}_{i,t},U^{n,2}_{i,t})_{i\in I,t>nB}\right),
$$
for $1\leq n\leq 2 k$ and for $n=2k+1$,
$$
\left((U^{n,1}_{i,t},U^{n,2}_{i,t})_{i\in I,t\leq (2k-1)B}, (U^{0,1}_{i,t},U^{0,2}_{i,t})_{i\in I,(2k-1)B<t \leq T}, (U^{n,1}_{i,t},U^{n,2}_{i,t})_{i\in I,t>T}\right).
$$

Since all chains are simulated with the same set of weights $(\lambda_i)_{i\in I}$ and transitions $(p_i^v)_{i\in I, v\in \cV}$, they have obviously the same distribution. Since the algorithms use disjoint sets of uniform variables for the odd (resp. even) chains, they are obviously independent and therefore Items 1-3 follows easily from the construction.

Let $G_{i,t}$ be the genealogy of site $(i,t)$ in the chain ${\bf X}$ and $\bT_{i,t}=\bT(G_{i,t})$. 
For any $n$, any $i\in F$ and any $t\in I_n$, if $\bT_{i,t}>(n-2) B$, then we use exactly the same set of uniform variables to produce the values of $X_{i,t}$ and $X^n_{i,t}$ and their values are equal. 

Therefore on $\Omega_{good}=\cap_{i\in F}\cap_{n=1}^{2k+1}\cap_{t\in I_n}\{\bT_{i,t}>(n-2)B\}$, $X_{F,I_n}=X^n_{F,I_n}$ for all $n=1,...,2k+1$. Note that $\Omega_{good}$ only depends on ${\bf X}$.

It remains to control $\tilde{\P}(\Omega^c_{good})$. 
By a union bound, and the application of Theorem \ref{laplace_gen}, we obtain
\begin{eqnarray*}
\tilde{\P}(\Omega^c_{good}) &\leq & \sum_{i\in F} \sum_{n=1}^{2k+1} \sum_{t\in I_n} \P(\bT_{i,t}\leq (n-2)B)\\
&\leq &\sum_{i\in F} \sum_{n=1}^{2k+1} \sum_{t\in I_n} \P(t-\bT_{i,t}\geq t- (n-2)B)\\
&\leq &\sum_{i\in F} \sum_{n=1}^{2k+1} \sum_{t\in I_n} e^{-\theta(t-(n-2)B)} \Psi( \theta)\\
&\leq & |F| (2k+1) \frac{e^{-\theta (B-m+1)}}{1-e^{-\theta}} \Psi(\theta).
\end{eqnarray*}

In particular if we choose $ B= m + \theta^{-1}( 2 \log(T) +\log(|F|)$,
$$\tilde{\P}(\Omega^c_{good}) \leq \frac{2k+1}{T^2} \frac{\Psi(\theta)}{1-e^{-\theta}},$$
which concludes the proof.

\subsection{Proof of Theorem \ref{Hoeff}}
Take $B = m +\theta^{-1} ( 2 \log(T)+ \log(|F|))$,  $k=\lfloor \frac{T}{2B} \rfloor$ and use the probability space $(\tilde{\Omega}, \tilde{\cF}, \tilde{\P})$
 and the stochastic chains ${\bf X},\ldots ,{\bf X^{2k+1}}$ given by Lemma \ref{Lemma:1}. 
By Lemma \ref{Lemma:1}-Item 1 we can assume that $Z$ is also defined on $(\tilde{\Omega}, \tilde{\cF}, \tilde{\P})$.
Define also a partition $J_1, \ldots, J_{2k+1}$  of $1:T$ as follows:
\begin{equation*}
J_n=\{1+(n-1)B, \ldots, nB\} \  \mbox{for} \ 1\leq n\leq 2k, \ \mbox{and} \  J_{2k+1}=\{1+2kB, \ldots, T\}.
\end{equation*}
For each $1\leq n\leq 2k+1$, write $
S_n=\frac{1}{T}\sum_{t\in J_n}f(X^{n}_{F,t-m:t-1})$ and note that $S_n$ only depends on the $t$'s in $I_n$ as defined in Lemma \ref{Lemma:1}.
Since $|J_n|\leq B$
for all $1\leq n\leq 2k+1,$ it holds $|S_n| \leq MB/T$.

Observe that Lemma \ref{Lemma:1}-Item 1 and 4 ensure that on $\Omega_{good},$
$$Z=\sum_{n=1}^{2k+1}(S_{n}-\E(S_n)),$$
so that for any $w>0$, we have
$$
\tilde{\P}(Z>w)\leq \tilde{\P}(\Omega^c_{good})+\tilde{\P}\left(\sum_{n=1}^{2k+1}(S_{n}-\E(S_n))>w\right) \leq \frac{c'(\theta)}{T} +\tilde{\P}\left(\sum_{n=1}^{2k+1}(S_{n}-\E(S_n))>w\right)  .
$$
Moreover, if we denote $Z_1=\sum_{n=1}^{k+1}(S_{2n-1}-\E(S_{2n-1}))$ and $Z_2=\sum_{n=1}^{k}(S_{2n}-\E(S_{2n})$, then
\begin{equation*}
\label{Ineq:Z_smaller_Z_1_plus_Z_2}
\tilde{\P}\left(\sum_{n=1}^{2k+1}(S_{n}-\E(S_n))>u+v\right)\leq \tilde{\P}\left(Z_1>u\right)+\tilde{\P}\left(Z_2>v\right),
\end{equation*}
for all $u+v=w$.

Lemma \ref{Lemma:1}-Item 3 implies that $S_2,\ldots, S_{2k}$ are independent, so that by the classical Hoeffding  inequality, we have for any $x>0$,
$
\tilde{\P}\left(Z_1>\sqrt{k B^2M^2 T^{-2}x/2}\right)\leq e^{-x},
$
and similarly for $\tilde{\P}\left(Z_1>\sqrt{(k+1) B^2M^2T^{-2} x/2}\right)\leq e^{-x}$. Hence
$$
\tilde{\P}\left(Z>\sqrt{k B^2M^2 T^{-2}x/2} +\sqrt{(k+1) B^2M^2T^{-2} x/2} \right)\leq \frac{c'(\theta)}{T} + 2e^{-x}.$$
But 
$k \leq T(2B)^{-1}$ and $k+1 \leq (T+2B)(2B)^{-1} \leq T/B$.
This leads directly to the first result. 

For the second result, note that we can restrict ourselves to $\Omega_{good}$  once and for all at the beginning and use the union bound only on the auxiliary independent chains, which explains why we pay $|\mathcal{F}|$ only in front of the deviation $e^{-x}$. 

\subsection{Proof of Theorem \ref{Hoeff_Matrices}}

Let $(\tilde{\Omega}, \tilde{\cF}, \tilde{\P})$ be the probability space and ${\bf X},\ldots ,{\bf X^{2k+1}}$ be the stochastic chains given by Lemma \ref{Lemma:1}. 
By Lemma \ref{Lemma:1}-Item 1 we can assume that $Z$ is also defined on $(\tilde{\Omega}, \tilde{\cF}, \tilde{\P})$. We write $\tilde{\E}$ to denote the expectation taken with respect the probability measure $\tilde{\P}$.

Now, let $B$, $k$, $J_1, \ldots, J_{2k+1}$ as in the proof of Theorem \ref{Hoeff} and define for $1\leq n\leq 2k+1$, the random matrix $\Sigma_n=((\Sigma_n(f,g))_{f,g\in\cF}$ as follows: 
\begin{equation*}
\Sigma_n(f,g)=\frac{1}{T}\sum_{t\in J_n}\left(f(X^{n}_{F,t-m:t-1})g(X^{n}_{F,t-m:t-1})-\E(f(X^{n}_{F,t-m:t-1})g(X^{n}_{F,t-m:t-1})\right).
\end{equation*}

Clearly $\tilde{\E}(\Sigma_n)=0$. To apply Theorem 1.3 of \cite{Tropp2012}, we need to find a deterministic  self-adjoint matrix $A_n$ such that
$A_n^2-\Sigma_n^2$ is  non negative. This means that for all vector $x\in \R^\cF$,
$$x^\intercal [A_n^2 -\Sigma_n^2] x \geq 0.$$
By taking $A_n= \sigma I_n$, it is sufficient to prove that 
$$x^\intercal \Sigma_n^2 x \leq \sigma^2 \|x\|^2.$$
But 
\begin{eqnarray*}
x^\intercal \Sigma_n^2 x  &=& \sum_{f,g \in \cF} x_f x_g\frac{1}{T^2} \sum_{t,t' \in J_n} \sum_{h \in \cF} \left(f(X^{n}_{F,t-m:t-1})h(X^{n}_{F,t-m:t-1})-\E(f(X^{n}_{F,t-m:t-1})h(X^{n}_{F,t-m:t-1})\right) \times \\
&&\left(h(X^{n}_{F,t'-m:t'-1})g(X^{n}_{F,t'-m:t'-1})-\E(h(X^{n}_{F,t'-m:t'-1})g(X^{n}_{F,t'-m:t'-1})\right)\\
&= & \frac{1}{T^2} \sum_{t,t' \in J_n} \sum_{h \in \cF} \left[\sum_f x_f \left(f(X^{n}_{F,t-m:t-1})h(X^{n}_{F,t-m:t-1})-\E(f(X^{n}_{F,t-m:t-1})h(X^{n}_{F,t-m:t-1})\right) \right] \times \\
&&\left[\sum_g x_g \left(g(X^{n}_{F,t'-m:t'-1})h(X^{n}_{F,t'-m:t'-1})-\E(g(X^{n}_{F,t'-m:t'-1})h(X^{n}_{F,t'-m:t'-1})\right) \right] \\
& \leq & \frac{1}{T^2} \sum_{t,t' \in J_n} \sum_{h \in \cF} \|x\|^2 \sqrt{\sum_f \left(f(X^{n}_{F,t-m:t-1})h(X^{n}_{F,t-m:t-1})-\E(f(X^{n}_{F,t-m:t-1})h(X^{n}_{F,t-m:t-1})\right)^2}\times\\
&& \sqrt{\sum_g \left(g(X^{n}_{F,t'-m:t'-1})h(X^{n}_{F,t'-m:t'-1})-\E(g(X^{n}_{F,t'-m:t'-1})h(X^{n}_{F,t'-m:t'-1})\right)^2}\\
&\leq & \frac{4\|x\|^2 |\cF|}{T^2} \sum_{t,t' \in J_n} \sum_{h \in \cF} M^4 \\
&\leq & \frac{4 |\cF|^2 B^2 M^4}{T^2} \|x\|^2.
\end{eqnarray*}

Hence $\sigma=\frac{2 |\cF| B M^2}{T}$ works.
Denote $Z_1=\sum_{n=1}^{k+1}\Sigma_{2n-1}$ and $Z_2=\sum_{n=1}^{k}\Sigma_{2n}$. Lemma \ref{Lemma:1} implies that on $\Omega_{good}$,
$$Z=Z_1+Z_2,$$
so that by the triangle inequality we have for any $u>0$ and $v>0$,
$$\tilde{\P}(\|Z\|>u+v)\leq \tilde{\P}(\Omega^c_{good})+\tilde{\P}(\|Z_1\|>u)+\tilde{\P}(\|Z_2\|>v).$$

Since by Lemma \ref{Lemma:1}-item 3, $\Sigma_2, \Sigma_4,\ldots, \Sigma_{2k}$ are i.i.d random matrices, we can apply Theorem 1.3 of \cite{Tropp2012} to deduce that for any $v>0,$
$$
\tilde{\P}\left(\|Z_2\|>\sqrt{8 k \sigma^2 v}\right)\leq 2|\cF|e^{-v},
$$
Similarly, we have that for any $x>0,$
$$
\tilde{\P}\left(\|Z_2\|>\sqrt{8 (k+1)\sigma^2u}\right)\leq 2|\cF|e^{-u},
$$
and as a consequence, it follows that for any $x>0,$ 
\begin{equation*}
\tilde{\P}\left(\|Z\|>\sqrt{8k\sigma^2x}+\sqrt{8(k+1)\sigma^2 x}\right)\leq \frac{c'(\theta)}{T}+4|\cF|e^{-x}.
\end{equation*}
Since $k^{1/2}+(k+1)^{1/2}\leq (4T/B)^{1/2}$, the result follows from the inequality above.

\subsection{Proof of Lemma \ref{likelihood_lower_bound}}
The proof is done for the lower bound. The argument is similar for the upper bound. We use induction on the time length of $v$. If $v=\emptyset$, $f$ is constant and  $\E(f(X_v))=\E_{\mathcal{B}(1/2}^{\otimes \mathcal{V}}(f(X_v))$.
Let $Q=\mathcal{B}(1/2)^{\otimes v}$.

If the time length of $v$ is strictly positive, let $t$ be the maximal time of $v$ and let $w_t=\{(i,t) \mbox{ for } i \mbox{ such that } (i,t)\in v\}$.
\begin{eqnarray*}
\E(f(X_v)) &= & \E[\E(f(X_v)|X_{-\infty:t-1})]\\
&=& \E\left(\sum_{x_{w_t} \in \{0,1\}^{w_t}} f((X_{v\setminus w_t},x_{w_t})) \P(X_{w_t}=x_{w_t}|X_{-\infty:t-1})\right) \\
&=& \E\left(\sum_{x_{w_t} \in \{0,1\}^{w_t}} f((X_{v\setminus w_t},x_{w_t}))  \prod_{i/ (i,t)\in w_t} \P(X_{i,t}=x_{i,t}|X_{-\infty:t-1})\right)\\
&\geq & (2\mu)^{|w_t|} \E\left(\sum_{x_{w_t} \in \{0,1\}^{w_t}} f((X_{v\setminus w_t},x_{w_t})) Q(X_{i,t}=x_{i,t})\right)\\
\end{eqnarray*}

But $\sum_{x_{w_t} \in \{0,1\}^{w_t}} f((X_{v\setminus w_t}),x_{w_t})) Q(X_{i,t}=x_{i,t})$ is a cylindrical function on $v\setminus w_t$ with time length strictly smaller than $v$, so by induction,
\begin{eqnarray*}
\E(f(X_v)) & \geq & (2\mu)^{|w_t|} (2\mu)^{|v\setminus w_t|} \E_{\mathcal{B}(1/2)}^{\otimes \mathcal{V}}\left(\sum_{x_{w_t} \in \{0,1\}^{w_t}} f((X_{v\setminus w_t}),x_{w_t})) Q(X_{i,t}=x_{i,t})\right) \\
&\geq & (2\mu)^{|v|} \E_{\mathcal{B}(1/2)}^{\otimes \mathcal{V}}(f(X_v)),
\end{eqnarray*}
which concludes the proof.

\subsection{Proof of Theorem \ref{inv(k)_High_Prob}}
For any $a\in\R^{\Phi}$ such that $\|a\|=1$, we have by Cauchy-Schwarz inequality
\begin{equation}
\label{Ineq:C_S}
\kappa \leq a^\intercal \E(G)a\leq a^\intercal G a + \|a\| \|(G-\E(G))a\|\leq a^\intercal G a+\|G-\E(G)\|,
\end{equation}
so that the result follows from Theorem \ref{Hoeff_Matrices} with $x=\log(4|\cF|/\delta)$ and $\cF=\Phi$.

\subsection{Proof of Theorem \ref{REHawkes}}
First of all, remark that thanks to Lemma \ref{likelihood_lower_bound} and since $\varphi$ in this case depends on a neighborhood of size 1, one has that
$$\E(G_{\varphi,\varphi}) = \E(\varphi(X)^2)\geq 2\mu 1/2=\mu$$
and similarly for $\varphi\not=\varphi'$, $\varphi\varphi'$ is positive and depends on a neighborhood of size 2, hence
$$ (1-\mu)^2\geq \E(G_{\varphi,\varphi'}) \geq \mu^2.$$

Moreover let us apply our version of Hoeffding's inequality, i.e. the second result of  Theorem \ref{Hoeff} on all the  $\varphi^2=\varphi$, $\varphi \varphi'$ and  $-\varphi \varphi'$ for $\varphi\not=\varphi'$.  Hence there exists and event of probability larger than $1- \frac{c'(\theta)}{T}- \delta$ such that for all $\varphi,\varphi'$,
$$ |G_{\varphi,\varphi'} - \E(G_{\varphi,\varphi'}) | \leq R_T,$$
with $$R_T = \sqrt{c"(\theta) \frac{(m+\log T + \log |F|)}{T} \log\left(\frac{4 |\Phi|^2}{\delta}\right)},$$
which means that there exists  a constant $c_1$ depending only on the distribution such that for $T$ large enough (depending on $\theta$ and $|F|$)
$$R_T = c_1 T^{-1} (m+\log T+ \log|F|)^{1/2}(\log m +\log |F| +\log \delta^{-1})^{1/2}.$$

Therefore on this event, for all  $a$ and $J$ such that $|J|\leq s$ and $|a_{J^c}|_1\leq c |a_J|_1$, and if $\mu^2 \geq R_T$, 
\begin{eqnarray*}
a^\intercal G a & = & \sum_{\varphi\in \Phi} a_\varphi^2 G_{\varphi,\varphi} + \sum_{\varphi \not = \varphi' \in \Phi} a_\varphi a_{\varphi'} G_{\varphi,\varphi'}\\
&\geq & (\mu-R_T) \sum_{\varphi\in \Phi} a_\varphi^2+ (\mu^2-R_T) \sum_{\begin{matrix}\varphi \not = \varphi' \in \Phi \\ a_\varphi a_{\varphi'} \geq 0\end{matrix}}a_\varphi a_{\varphi'} + ((1-\mu)^2+R_T) \sum_{\begin{matrix}\varphi \not = \varphi' \in \Phi \\ a_\varphi a_{\varphi'} <0\end{matrix}}a_\varphi a_{\varphi'} \\
&\geq & (\mu-\mu^2) \|a\|^2 + \mu^2 \sum_{\varphi, \varphi' \in \Phi}a_\varphi a_{\varphi'} + (1-2\mu) \sum_{\begin{matrix}\varphi \not = \varphi' \in \Phi \\ a_\varphi a_{\varphi'} <0\end{matrix}}a_\varphi a_{\varphi'}-R_T|a|^2_1 \\
&\geq & (\mu-\mu^2) \|a\|^2 + \mu^2 \left(\sum_{\varphi\in \Phi}a_\varphi\right)^2 - ((1-2\mu)-R_T) |a|_1^2 \\
& \geq & (\mu-\mu^2) \|a\|^2 - ((1-2\mu)+R_T) \left[ |a_J|_1+|a_{J^c}|_1\right]^2\\
& \geq &  (\mu-\mu^2) \|a\|^2- ((1-2\mu)+R_T)) (1+c)^2 |a_J|_1^2 \\
& \geq &  (\mu-\mu^2) \|a_J\|^2- ((1-2\mu)+R_T) (1+c) s \|a_J\|^2,\\
\end{eqnarray*}
which is the desired result.

\subsection{Proof of Corollary \ref{Oracle_HP_Short_effect}}
We shall prove only for the short effect dictionary. The other cases are treated similarly. 
For this choice of dictionary $\|\Phi\|_{\infty}=1$ and $|\Phi|=|F|$. Hence, by  applying Theorem \ref{inv(k)_High_Prob} and Theorem \ref{thm:Oracle_inequality} both with $\delta=T^{-1}$ one deduces that, for $T$ large enough, on an event of probability larger then 1-$c_1/T$, the following oracle inequality holds 
\begin{equation}
\label{oracle_prof_col1}
\| \hat{f}-p_i(\cdot) \|^2_T\leq   \inf_{a\in\R^{\Phi}}\left\{\|f_a-p_i(\cdot) \|^2_T + 4\kappa^{-1}|S(a)|\frac{(\log|F|+\log(2T))}{2T}\right\},
\end{equation}
where $c_1$ depends only on the distribution of {\bf X} and
$$
\kappa=\kappa'
-c'_1T^{-1/2}|F|^{1/2}(m+\log(T)+\log|F|)^{1/2}(\log|F|+\log\delta^{-1})^{1/2},
$$
with $c'_1$ depending only on the distribution of {\bf X} and
$\kappa'$ given by \eqref{kappa_prime_Short_effect}.

Now, for the choices given by \eqref{choixshort}, \eqref{choixcum} and \eqref{choixcumspont}, then, as seen previously $\kappa'=c'_2\log(T))^{-c_3'}$, for positive constants $c'_2$ and $c'_3$ depending only on $m$ and $\mu$ and 
$$\kappa=\frac{c'_2}{(\log T)^{c'_3}}(1-o(1)).$$
By plugging $\kappa$ into \eqref{oracle_prof_col1}, the result follows.

\section*{Acknowledgements}
This research has been conducted as part of FAPESP project {\em Research, Innovation and
Dissemination Center for Neuromathematics} (grant 2013/07699-0).
This work was also supported by the French government, through the UCA$^{Jedi}$ ”Investissements d’Avenir” managed by the National Research Agency (ANR-15-IDEX-01), by the structuring program $C@UCA$ of Universit\'e C\^ote d'Azur and by the interdisciplinary axis MTC-NSC of the University of Nice Sophia-Antipolis.
We would like to thank anonymous referees for their helpful comments that improved the manuscript.

\appendix


\bibliography{Bibli}{}
\bibliographystyle{imsart-nameyear}

\end{document}